\documentclass[11pt,leqno]{amsart}
\usepackage[utf8]{inputenc}

\usepackage{amsmath}
\usepackage{amsthm}
\usepackage{enumerate}
\usepackage{mathtools}
\usepackage{stmaryrd}
 
\usepackage[colorlinks=false]{hyperref}

\usepackage{color}

\usepackage[charter]{mathdesign}

\makeatletter
\g@addto@macro\bfseries{\boldmath}
\makeatother

\usepackage{xy}
\xyoption{all}

\numberwithin{equation}{section}

\theoremstyle{plain}
    \newtheorem{theorem}[equation]{Theorem}
    \newtheorem{lemma}[equation]{Lemma}
    \newtheorem{corollary}[equation]{Corollary}

    \newtheorem*{theorem*}{Theorem}
    \newtheorem*{proposition*}{Proposition}
    \newtheorem*{corollary*}{Corollary}
    \newtheorem*{lemma*}{Lemma}
    \newtheorem*{conjecture*}{Conjecture}
    \newtheorem{definition-theorem}[equation]{Definition/Theorem}
    \newtheorem{definition-lemma}[equation]{Definition/Lemma}
\theoremstyle{definition}
    \newtheorem{definition}[equation]{Definition}
    \newtheorem{example}[equation]{Example}
    \newtheorem{examples}[equation]{Examples}

    \newtheorem{remark}[equation]{Remark}
    \newtheorem{hypotheses}[equation]{Hypotheses}

    \newcommand{\C}{\mathbb{C}}

   	\renewcommand{\phi}{\varphi}
	\let\epsilon\varepsilon\let\smallepsilon\oldepsilon

    \newcommand{\Compact}{\operatorname{K}}
    
    \newcommand{\Multiplier}{\operatorname{M}}
    
    \newcommand{\alg}{}
    
   \newcommand{\functor}{\mathcal}

\newcommand{\into}{\hookrightarrow}
\newcommand{\onto}{\twoheadrightarrow}

\newcommand{\restrict}{\raisebox{-.5ex}{$|$}}

\newcommand{\id}{\mathrm{id}}

\newcommand{\dual}{\vee}
  \renewcommand{\prod}{\bigsqcap}

    \DeclareMathOperator{\Hom}{Hom}

    \DeclareMathOperator{\lspan}{span}
	\DeclareMathOperator{\image}{image}

 \newcommand{\h}{\mathrm{h}}

	\DeclareMathOperator{\CMod}{CM}
	\DeclareMathOperator{\CComod}{CC}
	
	\DeclareMathOperator{\Prim}{Prim}
	\newcommand{\Glue}{\operatorname{Glue}}

	\DeclareMathOperator{\Pic}{Pic}
	\DeclareMathOperator{\Centre}{Z}
	\DeclareMathOperator{\Eq}{Eq}

\title{Gluing Hilbert $C^*$-modules over the primitive ideal space}
\author{Tyrone Crisp }
\date{June 2020}
\address{Department of Mathematics \& Statistics, University of Maine.
5752 Neville Hall, Room 333.
Orono, ME 04469 USA}
\email{tyrone.crisp@maine.edu}



\begin{document}

\begin{abstract}
We show that the gluing construction for Hilbert modules introduced by Raeburn in his computation of the Picard group of a continuous-trace $C^*$-algebra (\textit{Trans. Amer. Math. Soc.}, 1981) can be applied to arbitrary $C^*$-algebras, via an algebraic argument with the Haagerup tensor product. We put this result into the context of descent theory by identifying categories of gluing data for Hilbert modules over $C^*$-algebras with categories of comodules over $C^*$-coalgebras, giving a Hilbert-module version of a standard construction from algebraic geometry. As a consequence we show that if two $C^*$-algebras have the same primitive ideal space $T$, and are Morita equivalent up to a $2$-cocycle on $T$, then their Picard groups relative to $T$ are isomorphic.
\end{abstract}

\maketitle

\section{Introduction}\label{sec:intro}

In \cite{Crisp-descent} we showed how the categorical technique of {comonadic descent} can be applied to yield equivalences between categories of Hilbert modules over $C^*$-algebras on the one hand, and categories of Hilbert comodules over a certain kind of $C^*$-coalgebras on the other. Several  interesting invariants of $C^*$-algebras---for instance, Morita equivalence classes, $K$-theory, and Picard groups---can be defined   in terms of the category of Hilbert modules, and so the equivalence between modules and comodules  gives a new point of view from which to study these invariants.

The usefulness of this approach depends, of course, upon a good understanding of the coalgebras and comodules in question. These are defined in terms of the Haagerup tensor product, whose strong universal properties make it, of necessity, rather abstract. The purpose of this paper is to present an example in which Hilbert comodules can be described explicitly in terms of standard Hilbert-module theory, and to demonstrate how the equivalence between modules and comodules gives rise to recognisably $C^*$-algebraic results. 

Our results generalise work of Raeburn \cite{Raeburn-Picard} on continuous-trace $C^*$-algebras. Let $A$ be an arbitrary $C^*$-algebra, and let $ (F_i)_{i\in I}$ be a locally finite closed cover of the primitive ideal space of $A$. The inclusion of $C^*$-algebras $A\into \bigoplus_{i\in I}A\restrict_{F_i}$ yields, via the machinery of \cite{Crisp-descent}, a $C^*$-coalgebra $C$ and a unitary equivalence  between the category of Hilbert modules over $A$ and the category of Hilbert comodules over $C$. We will prove that this comodule category can be identified, via a conceptually straightforward but technically somewhat fiddly construction, with a category of gluing data for Hilbert modules, an object of the latter category being a collection $(Z_i)_{i\in I}$ of Hilbert modules over the quotients $A\restrict_{F_i}$, together with isomorphisms $Z_j\restrict_{F_i\cap F_j}\xrightarrow{\cong} Z_i\restrict_{F_i\cap F_j}$ over the intersections $F_i\cap F_j$ satisfying a natural associativity condition on triple intersections. The equivalence between modules and comodules implies that every such gluing datum can be assembled into a Hilbert module over $A$, and that two gluing data are isomorphic if and only if the glued modules are isomorphic. Applying this to Morita equivalences yields a generalisation of (part of) \cite[Theorem 2.1]{Raeburn-Picard} to arbitrary $C^*$-algebras. (We note that this is not the same generalisation as the one alluded to in \cite[p.197, Remark]{Raeburn-Picard}: for one thing, the Morita equivalences occurring in our results all act trivially on the primitive ideal space, cf.~\cite[p.192, Remark]{Raeburn-Picard}.)

The construction through which a gluing datum becomes a Hilbert module is the same as the one appearing in \cite{Raeburn-Picard} (and it is analogous to similar gluing constructions appearing in many other contexts). Our method for proving that this construction yields an equivalence of categories is, however, quite different to that of \cite{Raeburn-Picard}: our proof relies on the  exactness properties of the Haagerup tensor product established by Anantharaman-Delaroche and Pop \cite{An-Pop}, together with some algebraic manipulations borrowed from the parallel algebraic-geometric setting of descent of along Zariski coverings (as explained for instance in \cite[Section 6.2]{BLR}), and some techniques from \cite{Crisp-descent}. The Haagerup tensor product does not commute with direct sums, and this necessitates some contortions when adapting the algebraic arguments to the $C^*$-algebraic setting. A noteworthy feature of our proof is that the question of whether or not the primitive ideal space is Hausdorff never arises.

The paper is organised as follows. In Section \ref{sec:basic} we set up the basic definitions and state our first main result, Theorem \ref{thm:main}. In Section \ref{sec:technical} we briefly review some background on operator modules and the Haagerup tensor product, and we establish some technical results that will be used in the proof of Theorem \ref{thm:main}. The proof itself occupies most of Section \ref{sec:gluing}. We have tried to make this paper readable independently of the abstract machinery of \cite{Crisp-descent}, and this has led to some redundancy: specifically, Lemmas \ref{lem:image-eta} and \ref{lem:kernels} are special cases of results from \cite{Crisp-descent}, and we have reproduced  the proofs here in somewhat abridged form. For the same reason we have avoided mention of coalgebras and comodules in most of the paper, the exception being Section \ref{subsec:comodules} in which we explain how Theorem \ref{thm:main} and its proof relate to the module-comodule equivalences of \cite{Crisp-descent}. Finally, in Section \ref{sec:Picard} we apply Theorem \ref{thm:main} to obtain our generalisation (Corollary \ref{cor:Picard}) of Raeburn's theorem on the Picard group.

\section{Pulling apart Hilbert modules}\label{sec:basic}

\begin{hypotheses}\label{hyp} Throughout the paper $A$ is a $C^*$-algebra, and $\mathfrak F=(F_i)_{i\in I}$ is locally finite closed covering of the primitive ideal space $\Prim A$. That is, $\mathfrak F$ is a family of subsets of $\Prim A$ having the following properties:
\begin{enumerate}[\rm(a)]
\item $\mathfrak F$ is a \emph{covering} of $\Prim A$: that is, $\Prim A=\bigcup_{i\in I} F_i$. 
\item Each $F_i$ is \emph{closed} in the Jacobson topology; thus for each $i\in I$ the intersection $J_i\coloneqq \bigcap_{J\in F_i} J$ is a closed, two-sided ideal of $A$. 
\item The covering $\mathfrak{F}$ is \emph{locally finite}: each $J\in \Prim A$ has an open neighbourhood $U$ with $U\cap F_i=\emptyset$ for all but finitely many $i$. 
\end{enumerate}
\end{hypotheses}

We shall assume that the reader is familiar with the basic language and theory of Hilbert modules (by which we mean modules with a $C^*$-algebra-valued inner product, and not Hilbert-space representations). Expositions of the theory can be found in \cite{RW}, \cite{Lance}, and \cite[Chapter 8]{BLM}, for example. Our Hilbert modules will always be right modules, except when we come to discuss Morita equivalences in Section \ref{sec:Picard}.

\begin{definition}\label{def:restriction}
For each right Hilbert $A$-module  $X$, and each $i\in I$, the subset 
\[
XJ_i\coloneqq \{ xa\in X\ |\ x\in X,\ a\in J_i\}
\]
is a norm-closed $A$-submodule of $X$ (by the Cohen-Hewitt factorisation theorem; cf.~\cite[Proposition 2.33]{RW}), and we let $X\restrict_{F_i}$ denote the quotient module:
\[
X\restrict_{F_i}\coloneqq X/ (XJ_i).
\]
The quotient mapping $X\onto X\restrict_{F_i}$ is denoted by $x\mapsto x\restrict_{F_i}$. This applies in particular to $X=A$; the quotients $A\restrict_{F_i}=A/J_i$ are $C^*$-algebras.

For each bounded $A$-module map $\alpha:X\to Y$ we denote by $\alpha\restrict_{F_i}$ the induced map $\alpha\restrict_{F_i}:X\restrict_{F_i}\to Y\restrict_{F_i}$ defined by $\alpha\restrict_{F_i}(x\restrict_{F_i}) \coloneqq \left( \alpha(x)\right)\restrict_{F_i}$.

We may regard $X\restrict_{F_i}$ either as an $A$-module, or as an $A\restrict_{F_i}$-module. Note that $X\restrict_{F_i}$ is not a Hilbert module over $A$, but it is a Hilbert module over $A\restrict_{F_i}$: the inner product is
\[
\left\langle \left. x\restrict_{F_i}\, \right| \, x'\restrict_{F_i} \right\rangle_{A\restrict_{F_i}} \coloneqq \left( \langle x | x'\rangle_A\right)\restrict_{F_i}.
\]
This Hilbert $A\restrict_{F_i}$-module is unitarily isomorphic to the Hilbert-module tensor product $X\otimes^{*}_A A\restrict_{F_i}$, via the mapping
\[
X\otimes^{*}_A A\restrict_{F_i} \xrightarrow{ x\otimes a\restrict_{F_i}\mapsto (xa)\restrict_{F_i}} X\restrict_{F_i}.
\] 
If $\alpha:X\to Y$ is an adjointable map of Hilbert $A$-modules, then $\alpha\restrict_{F_i}:X\restrict_{F_i}\to Y\restrict_{F_i}$ is an adjointable map of Hilbert $A\restrict_{F_i}$-modules, with $(\alpha\restrict_{F_i})^* = \alpha^*\restrict_{F_i}$.

All of these definitions also apply to intersections of two or more $F_i$s. To save space we write $F_{ij}\coloneqq F_i\cap F_j$, etc. Thus, for instance, if $X$ is a Hilbert $A$-module then for each $(i,j)\in I^2$ the quotient $X\restrict_{F_{ij}} = X/(X(J_i+J_j))$ is a Hilbert $A\restrict_{F_{ij}}$-module. Likewise, if $Y$ is a Hilbert $A\restrict_{F_i}$-module, then $Y\restrict_{F_{ij}} = Y/ (Y (J_i+J_j)\restrict_{F_i})$ is a Hilbert $A\restrict_{F_{ij}}$-module.
\end{definition}

\begin{remark}\label{rem:quotients}
The assertion that $X\restrict_{F_i}$ is a Hilbert module over $A\restrict_{F_i}$ includes the somewhat subtle fact that the norm induced on the quotient module $X\restrict_{F_i}$ by the given inner product is equal to the quotient norm; see \cite[Proposition 3.25]{RW}. Another proof of this fact, using operator modules and the Haagerup tensor product, is given in Remark \ref{rem:quotients-2}.
\end{remark}

Throughout this paper we shall use the language of $C^*$-categories, as explained in \cite[Section 1]{GLR}.

\begin{definition}
We let $\CMod(A)$ denote the category of right Hilbert $A$-modules, with adjointable maps as morphisms. We equip $\CMod(A)$ with its canonical $C^*$-category structure: each  morphism space $\Hom_{\CMod(A)}(X,Y)$ is given the operator norm, and the $*$-operation $\Hom_{\CMod(A)}(X,Y)\to \Hom_{\CMod(A)}(Y,X)$ is given by taking adjoints. 
\end{definition}

\begin{definition}\label{def:Glue}
Let $\mathfrak F$ be a locally finite closed covering of $\Prim A$. We define a $C^*$-category $\Glue(A,\mathfrak F)$ of \emph{gluing data for Hilbert modules} as follows. An object of $\Glue(A,\mathfrak F)$ is a pair $(Z,\zeta)$ where:
\begin{enumerate}[$\square$]
\item $Z=(Z_i)_{i\in I}$ is a collection of Hilbert modules: each $Z_i$ is a right Hilbert module over $A\restrict_{F_i}$.
\item $\zeta = (\zeta_{ij})_{(i,j)\in I^2}$ is a collection of unitary isomorphisms $\zeta_{ij}:Z_j\restrict_{F_{ij}}\xrightarrow{\cong} Z_i\restrict_{F_{ij}}$ of Hilbert $A\restrict_{F_{ij}}$-modules.
\item For each triple $(i,j,k)\in I^3$ we have an equality of maps $Z_k\restrict_{F_{ijk}}\to Z_i\restrict_{F_{ijk}}$,
\begin{equation*}\label{eq:zeta-assoc}
\zeta_{ik}\restrict_{F_{ijk}} = \zeta_{ij}\restrict_{F_{ijk}}\circ \zeta_{jk}\restrict_{F_{ijk}}.
\end{equation*}
\end{enumerate}
A morphism $(Z,\zeta)\to (W,\omega)$ in $\Glue(A,\mathfrak F)$ is a family $\alpha=(\alpha_i)_{i\in I}$ of adjointable maps $\alpha_i:Z_i\to W_i$ of Hilbert $A\restrict_{F_i}$-modules, such that $\sup_{i\in I}\|\alpha_i\|<\infty$, and such that
\[
\alpha_i\restrict_{F_{ij}}\circ \zeta_{ij} = \omega_{ij}\circ \alpha_j\restrict_{F_{ij}}
\] 
for all $(i,j)\in I^2$. We give $\Hom( (Z,\zeta), (W,\omega))$  the norm 
$
\lVert (\alpha_i)_{i\in I} \rVert \coloneqq \sup_{i\in I} \lVert \alpha_i \rVert,
$
and we define a $*$-operation 
\[
\ast: \Hom((Z,\zeta),(W,\omega))\to \Hom((W,\omega),(Z,\zeta))
\]
 by $(\alpha_i)_{i\in I}^*\coloneqq (\alpha_i^*)_{i\in I}$. It is easily checked these definitions make $\Glue(A,\mathfrak F)$ into a $C^*$-category.
\end{definition}

\begin{remark} 
Note that for each object $(Z,\zeta)$ of $\Glue(A,\mathfrak F)$, and for each $i\in I$, we have $\zeta_{ii}=\id_{Z_i}$. Indeed, $\zeta_{ii}$ is assumed to be unitary, and the identity $\zeta_{ik}\restrict_{F_{ijk}} = \zeta_{ij}\restrict_{F_{ijk}}\circ \zeta_{jk}\restrict_{F_{ijk}}$ for $i=j=k$ says that $\zeta_{ii}$ is an idempotent. This implies in turn that $\zeta_{ij}^*=\zeta_{ji}$ for all $(i,j)\in I$.
\end{remark}

Every Hilbert module over $A$ can be pulled apart to give a gluing datum:

\begin{definition}\label{def:P}
For each Hilbert $A$-module $X$, and each $(i,j)\in I^2$, we denote by $\kappa^X_{ij}:X\restrict_{F_j}\restrict_{F_{ij}}\to X\restrict_{F_i}\restrict_{F_{ij}}$ the canonical unitary isomorphism,
\[
\kappa^X_{ij}: \left(X\restrict_{F_j}\right) \restrict_{F_{ij}} \xrightarrow{ (x\restrict_{F_j})\restrict_{F_{ij}} \mapsto (x\restrict_{F_i})\restrict_{F_{ij}}} \left( X\restrict_{F_j}\right)\restrict_{F_{ij}}.
\]
We then define a $*$-functor $\functor{P}:\CMod(A)\to \Glue(A,\mathfrak F)$ as follows: for each Hilbert $A$-module $X$ we set
\[
\functor{P} X \coloneqq \left( \left(X\restrict_{F_i}\right)_{i\in I}, \left(\kappa^X_{ij}\right)_{(i,j)\in I^2} \right)
\]
and for each adjointable map $\alpha:X\to Y$ of Hilbert $A$-modules we set
\[
\functor{P}\alpha \coloneqq \left( \alpha \restrict_{F_i}\right)_{i\in I} : \functor{P}X\to \functor{P}Y.
\]
\end{definition}

\begin{theorem}\label{thm:main}
The functor $\functor{P}:\CMod(A)\to \Glue(A,\mathfrak F)$ is a unitary equivalence of $C^*$-categories. \end{theorem}

That is to say, there is a $*$-functor $\functor{G}:\Glue(A,\mathfrak F)\to \CMod(A)$ and unitary natural isomorphisms $\functor{G}\functor{P}\cong \id_{\CMod(A)}$ and $\functor{P}\functor{G}\cong \id_{\Glue(A,\mathfrak F)}$. Such a functor $\functor{G}$ is defined in Section \ref{subsec:G}, and the proof that $\functor{P}$ and $\functor{G}$ are mutually inverse equivalences occupies Sections \ref{subsec:pf-begin}--\ref{subsec:pf-end}.

\begin{remark}
Theorem \ref{thm:main}, and the proof that we shall give here, are modeled after a basic property of sheaves in algebraic geometry (see \cite[Section 6.2]{BLR}, for instance). It would be interesting to see how our methods fit into the framework of sheaves of $C^*$-algebras and operator modules of Ara and Mathieu \cite{Ara-Mathieu}. 
\end{remark}

\section{Some technical preliminaries}\label{sec:technical}

We continue to assume Hypotheses \ref{hyp}: $A$ is a $C^*$-algebra, and $\mathfrak F= (F_i)_{i\in I}$ is a locally finite closed covering of $\Prim A$.

\subsection{Direct sums of $C^*$-algebras and Hilbert modules}

Throughout this paper we use $\bigoplus$ to denote the $C_0$-direct sum of Banach spaces:
\[
\bigoplus_{i\in I} X_i =\left\{ \left. (x_i)_{i\in I}\in \prod_{i\in I}X_i \ \right|\ \|x_i\| \to 0 \text{ as }i\to\infty\right\}.
\]

\begin{definition}\label{def:B}
Let $B$ be the $C^*$-algebraic direct sum $B=\bigoplus_{i\in I} A\restrict_{F_i}$. Each $A\restrict_{F_i}$ is, in an obvious way, both a subalgebra and a quotient of $B$. We let $\pi^B_i:B\to A\restrict_{F_i}$ be the quotient mapping; for $b\in B$ we also write $b_i \coloneqq \pi^B_i(b)$.  For each $i\in I$ we denote by $1\restrict_{F_i}$ the unit element in the multiplier algebra $\Multiplier(A\restrict_{F_i})$. Using the embedding $A\restrict_{F_i}\into B$ we also regard $1\restrict_{F_i}$ as an element of $\Multiplier(B)$.
\end{definition}

\begin{lemma}\label{lem:A-in-B}
The map $\eta: a\mapsto (a\restrict_{F_i})_{i\in I}$ embeds $A$ as a nondegenerate subalgebra of $B$, and there is a contractive completely positive $A$-bimodule map $\iota:B\to A^{\dual\dual}$ such that the composition $\iota\circ\eta:A\to A^{\dual\dual}$ is the canonical embedding of $A$ into its second dual.
\end{lemma}

We will often suppress the map $\eta$ and just consider $A$ as a subalgebra of $B$.

\begin{proof}
To see that the image of $\eta$ is contained in the $C_0$-direct sum, recall (e.g., from \cite[Lemma A.30]{RW}) that for each $a\in A$ the function $\Prim A \to [0,\infty)$ defined by $J\mapsto \| a\restrict_{\{J\}}\|$ vanishes at infinity. Since the covering $\mathfrak F$ is locally finite, the function $i\mapsto \| a\restrict_{F_i}\|$ likewise vanishes at infinity. The  $*$-homomorphism $\eta$ is injective because the $F_i$ cover $\Prim A$.

For the nondegeneracy we are asserting that $\eta$ sends an approximate unit for $A$ to an approximate unit for $B$; this follows easily from the fact that each element of $B$ can be approximated by elements of finite support.

Finally,  the existence of a map $\iota$ with the given properties is known, by a result of Kirchberg \cite[Proposition 3.1]{Kirchberg-nonsemisplit}, to be equivalent to the assertion that for every $C^*$-algebra $C$ the inclusion $A\otimes^{\alg}_{\C} C \into B\otimes^{\alg}_{\C} C$ is isometric for the maximal tensor-product norm.  This latter property holds true for our $A$ and $B$ because the maximal tensor product commutes with direct sums, and because each of the mappings $A\otimes^{\alg}_{\C} C \to A\restrict_{F_j}\otimes^{\alg}_{\C} C$ is contractive for the $\max$ norm.
\end{proof}

\begin{definition}\label{def:Zi}
For each Hilbert $B$-module $Z$ and each $i\in I$ we define $Z_i\coloneqq Z\cdot 1\restrict_{F_i}$. This $Z_i$ is, in a natural way, a Hilbert $A\restrict_{F_i}$-module. The projection $Z\to Z_i$ is denoted by $\pi^Z_i$; we also write $z_i\coloneqq \pi^Z_i(z)$. In the other direction,  if $(Z_i)_{i\in I}$ is a collection of Hilbert $A\restrict_{F_i}$-modules, then the $C_0$-direct sum $Z=\bigoplus_{i\in I} Z_i$ is a Hilbert $B$-module. These constructions extend to a pair of mutually inverse unitary equivalences of $C^*$-categories, between the category $\CMod(B)$ of Hilbert $B$-modules and adjointable maps, and the category whose objects are collections $(Z_i)_{i\in I}$ of Hilbert $A\restrict_{F_i}$-modules, and whose morphisms are collections $(\alpha_i)_{i\in I}$ of adjointable $A\restrict_{F_i}$-module maps satisfying $\sup_{i\in I} \|\alpha_i\|<\infty$.
\end{definition}

For the next lemma---whose simple proof we omit---recall  that if $X$ is a Hilbert module over $A$ then the Hilbert-module tensor product $X\otimes^{*}_A B$ is the completion of the algebraic tensor product  $X\otimes^{\alg}_A B$ in the norm derived from the $B$-valued inner product $\langle x\otimes b \,|\,  x'\otimes b'\rangle_B \coloneqq b^*\eta(\langle x|x'\rangle_A) b$. This is a Hilbert $B$-module. 

\begin{lemma}\label{lem:X-tensor-B}
If $X$ is a Hilbert module over $A$ then the map 
\[
\Psi^X : X\otimes^{*}_A B \to \bigoplus_{i\in I} X\restrict_{F_i}, \qquad x\otimes b \mapsto \left( x\restrict_{F_i} b_i\right)_{i\in I}
\]
is a unitary isomorphism of Hilbert $B$-modules. \hfill\qed
\end{lemma}

\subsection{Operator modules and the Haagerup tensor product}

Let us briefly recall the principal facts about operator modules and the Haagerup tensor product that we shall use. See \cite{BLM} for a fuller account.

By an \emph{operator module} over a $C^*$-algebra we shall always mean a nondegenerate operator module, and we shall mean a right module unless otherwise specified. Here are the examples of operator modules that will appear in this paper.

\begin{examples}
\begin{enumerate}[\rm(a)]
\item If $C$ is a $C^*$-subalgebra of a $C^*$-algebra $D$, and if $Y$ is a norm-closed linear subspace of $D$ satisfying $YC=Y$, then $Y$ is an operator module over $C$. 
\item If $Z$ is a Hilbert module over a $C^*$-algebra $C$, then $Z$ is an operator module over $C$, \emph{and} over every nondegenerate $C^*$-subalgebra of $C$, \emph{and} over every $C^*$-algebra having $C$ as a quotient. Every adjointable map of Hilbert  modules is completely bounded as a map of operator modules, and every unitary isomorphism of Hilbert  modules is a completely isometric isomorphism of operator modules. (Here `completely' refers to the canonical norms on the spaces of matrices over an operator module.)
\item If $X$ is an operator $C$-module, and if $X'$ is a closed $C$-submodule of $X$, then the quotient $X/X'$ is an operator $C$-module.
\item In terms of the notation already established in this paper: each Hilbert $B$-module is an operator $A$-module and an operator $B$-module; and each Hilbert $A\restrict_{F_i}$-module is an operator $A$-module and an operator $B$-module. 
\end{enumerate}
\end{examples}

We now turn to the Haagerup tensor product. If $C$ and $D$ are $C^*$-algebras, and if $X$ is a right operator $C$-module and $Y$ an operator $C$-$D$-bimodule, then the \emph{Haagerup tensor product} $X\otimes^{\h}_C Y$ is a right operator $D$-module. Here are the main facts about $\otimes^{\h}$ that we shall use:

\begin{theorem}
\label{thm:Haagerup}
\begin{enumerate}[\rm(a)]
\item For each (nondegenerate, right) operator $C$-module $X$ the map $X\otimes^{\h}_C C\xrightarrow{x\otimes c\mapsto xc} X$ is a completely isometric isomorphism; and similarly for left modules.
\item If $X$ is a Hilbert $C$-module, and if $Y$ is a Hilbert $D$-module equipped with a nondegenerate homomorphism $C\to \Hom_{\CMod(D)}(Y,Y)$, then $X\otimes^{\h}_C Y$ is completely isometrically isomorphic to the Hilbert-module tensor product $X\otimes^{\ast}_C Y$.
\item If $\phi:X'\into X$ is a completely isometric map of operator $C$-modules, and if $\psi:X\to X/\phi(X')$ is the quotient mapping, then for every $Y$ the mapping $\phi\otimes \id:X'\otimes^{\h}_C Y\to X\otimes^{\h}_C Y$ is completely isometric; the image of this map is the kernel of the mapping $\psi\otimes \id : X\otimes^{\h}_C Y \to (X/\phi(X'))\otimes^{\h}_C Y$; and the latter map induces a completely isometric isomorphism 
\[
\left(X\otimes^{\h}_C Y\right )\big/\left( (\phi\otimes\id_Y)(X'\otimes^{\h}_C Y)\right) \xrightarrow{\cong} \left(X/\phi(X')\right)\otimes^{\h}_C Y.
\]
Likewise for complete isometries $\phi:Y'\into Y$ of left operator $C$-modules.
\item Suppose that $C$ is a $C^*$-subalgebra of $L$, and that $X\subseteq L$ is a norm-closed subspace satisfying $XC=C$. Suppose further that $C$ and $D$ are $C^*$-subalgebras of $R$, and that $Y\subseteq R$ is a norm-closed subspace satisfying $Y=CY=YD$. Then the map
\[
X\otimes^{\h}_C Y \to L\ast_C R,\qquad x\otimes y\mapsto x y
\]
is a completely isometric embedding into the amalgamated free product $C^*$-algebra.
\end{enumerate}
\end{theorem}

\begin{proof}
Part (a) is either a theorem (essentially \cite[Corollary 3.3]{CES}) or a definition, depending on how one defines operator modules. Part (b) is due to Blecher \cite[Theorem 4.3]{Blecher-newapproach}. Part (c) is due to Anantharaman-Delaroche and Pop \cite[\S7, Corollary]{An-Pop}. Part (d) is due, in various forms, to Christensen-Effros-Sinclair \cite[Theorem 3.1]{CES}, Pisier \cite[Lemma 1.14]{Pisier-Kirchberg}, and Ozawa \cite[p.515]{Ozawa}. 
\end{proof}

\begin{remark}\label{rem:quotients-2}
Following up on Remark \ref{rem:quotients}, let us give an example of how the extra flexibility afforded by operator modules (especially the existence of quotient modules), coupled with the strong exactness properties of the Haagerup tensor product, can yield conceptual simplifications of Hilbert-module arguments. Let $C$ be a $C^*$-algebra, let $J$ be an ideal of $C$, and let $X$ be a Hilbert $C$-module. Then the quotient norm on $X/XJ$ is the one induced by the $C/J$-valued inner product $\langle x+XJ \, |\, x'+XJ\rangle_{C/J}\coloneqq \langle x|x'\rangle_C +J$. Indeed, parts (a), (c), and (b) (in that order) of Theorem \ref{thm:Haagerup} give completely isometric isomorphisms
\[
X/XJ \cong \left( X\otimes^{\h}_C C\right)\big/ \left( X\otimes^{\h}_C J\right) \cong X\otimes^{\h}_C (C/J) \cong X\otimes^{*}_C (C/J)
\]
and it is easy to check that the $C/J$-valued inner product on $X\otimes^{*}_C (C/J)$ corresponds under these isomorphisms to the given inner product on $X/XJ$.
\end{remark}

Resuming our consideration of $A$, $B$, $\mathfrak F$, etc, as defined above, we shall now prove a series of technical lemmas that will be used in the proof of Theorem \ref{thm:main}. Most of the work is necessitated by the fact that the Haagerup tensor product does not commute with direct sums.

\begin{lemma}\label{lem:nu}
If $Y$ is a Hilbert module over $A\restrict_{F_i}$  then for each $j\in I$ the map
\[
\nu^Y_{ij}:Y\otimes^{\h}_A A\restrict_{F_j} \xrightarrow{\cong} Y\restrict_{F_{ij}},\qquad y\otimes a\restrict_{F_j} \mapsto y\restrict_{F_{ij}} \cdot a\restrict_{F_{ij}}
\]
is a completely isometric isomorphism of operator $A$-modules.
\end{lemma}

\begin{proof}
The $A$-linearity of $\nu^Y_{ij}$ is obvious. To see that this map is a completely isometric isomorphism, factor $\nu^Y_{ij}$ as the composition
\begin{equation}\label{eq:nu-proof}
Y\otimes^{\h}_A A\restrict_{F_{j}} \xrightarrow{\restrict_{F_{ij}}\otimes \id} Y\restrict_{F_{ij}}\otimes^{\h}_A A\restrict_{F_j} \xrightarrow{ \id \otimes \restrict_{F_{ij}}} Y\restrict_{F_{ij}}\otimes^{\h}_A A\restrict_{F_{ij}}  \xrightarrow{y\otimes a\mapsto ya} Y\restrict_{F_{ij}}.
\end{equation}
Since the restriction mapping $\restrict_{F_{ij}}:Y\to Y\restrict_{F_{ij}}$ is a quotient mapping with kernel $YJ_{j}\restrict_{F_i}$, part (c) of Theorem \ref{thm:Haagerup} implies that $\restrict_{F_{ij}}\otimes \id : Y\otimes^{\h}_A A\restrict_{F_j}\to Y\restrict_{F_{ij}}\otimes^{\h}_A A\restrict_{F_j}$ is a quotient mapping with kernel $YJ_{j}\restrict_{F_{i}}\otimes^{\h}_A A\restrict_{F_j}$. Since $J_j$ acts by zero on $A\restrict_{F_j}$, this kernel is zero, and so the first map in \eqref{eq:nu-proof} is a completely isometric isomorphism. A similar argument shows that the second map in \eqref{eq:nu-proof} is a completely isometric isomorphism. The fact that the third map in \eqref{eq:nu-proof} is a completely isometric isomorphism is part (a) of Theorem \ref{thm:Haagerup}.
\end{proof}

\begin{lemma}\label{lem:mu}
Let $Z=\bigoplus_{i\in I} Z_i$ be a Hilbert $B$-module. The maps
\[
\mu^Z_{ij}: Z\otimes^{\h}_A B \to Z_i\restrict_{F_{ij}},\qquad z\otimes b \mapsto z_i\restrict_{F_{ij}} \cdot b_j\restrict_{F_{ij}}
\]
separate the points of $Z\otimes^{\h}_A B$, as $(i,j)$ ranges over $I^2$. Similarly, the maps
\[
\mu^Z_{ijk}: Z\otimes^{\h}_A B\otimes^{\h}_A B \to Z_i\restrict_{F_{ijk}},\qquad z\otimes b\otimes b' \mapsto z_i \restrict_{F_{ijk}} \cdot b_j\restrict_{F_{ijk}} \cdot b'_k\restrict_{F_{ijk}}
\]
separate the points of $Z\otimes^{\h}_A B\otimes^{\h}_A B$, as $(i,j,k)$ ranges over $I^3$.
\end{lemma}

\begin{proof}
Lemma \ref{lem:nu} implies that the maps 
\[
Z_i\otimes^{\h}_A A\restrict_{F_j} \to Z_i\restrict_{F_{ij}},\qquad z_i \otimes b_j \mapsto z_i \restrict_{F_{ij}}\cdot b_j\restrict_{F_{ij}}
\]
are completely isometric isomorphisms, and  so to show that the maps $\mu^Z_{ij}$ separate points it will suffice to show that the maps $\pi^Z_i\otimes \pi^B_j:Z\otimes^{\h}_A B\to Z_i\otimes^{\h}_A A\restrict_{F_j}$ separate points. 

To do this, we will first show that the maps
$
\pi^Z_i\otimes \id_B : Z\otimes^{\h}_A B\to Z_i\otimes^{\h}_A B
$
separate points, as $i$ ranges over $I$. Suppose that $t\in Z\otimes^{\h}_A B$ has $(\pi^Z_i\otimes \id_B)(t)=0$ for every $i\in I$. Fix $\smallepsilon>0$ and find a finite subset $S\subseteq I$ and an element $s\in (\bigoplus_{i\in S} Z)\otimes^{\h}_A B$ with $\|t - s\|<\epsilon$. Let $\pi^Z_{I\setminus S}:Z\to \bigoplus_{i\in I\setminus S} Z_i\subseteq Z$ be the projection. We have
\[
\begin{aligned}
\|t\| & =\left \| \left(\pi^Z_{I\setminus S} + \sum_{i\in S} \pi^Z_i\right)\otimes \id_B (t) \right\| \\
& \leq \left\| \pi^Z_{I\setminus S}\otimes \id_B(t-s) \right\| + \left\| \pi^Z_{I\setminus S}\otimes \id_B(s) \right\| + \sum_{i\in S} \left\| \pi^Z_i\otimes \id_B(t) \right\| \\
&<\smallepsilon + 0 + 0
\end{aligned}
\]
and so $t=0$. 

One shows similarly that for each  $i\in I$ the maps 
$
\id_{Z_i}\otimes \pi^B_j: Z_i\otimes^{\h}_A B\to Z_i\otimes^{\h}_A A\restrict_{F_j}
$
separate points as $j$ ranges over $I$, and it follows that the composite maps $\pi^Z_i\otimes \pi^B_j = (\id\otimes \pi^B_j)\circ (\pi^Z_i\otimes \id)$ separate points as $(i,j)$ ranges over $I$.

This shows that the $\mu^Z_{ij}$s separate points, and an analogous argument applies to the $\mu^Z_{ijk}$s.
\end{proof}

\begin{definition}\label{def:phi} 
For each Hilbert $B$-module $Z$, and for each $(i,j)\in I^2$, we consider the completely isometric $A\restrict_i$-linear map $\phi^Z_{ij}:Z_i\restrict_{ij}\into Z\otimes^{\h}_A B$ defined as the composition
\[
\phi^Z_{ij}:Z_i\restrict_{F_{ij}} \xrightarrow{\left(\nu^{Z_i}_{ij}\right)^{-1}} Z_i\otimes^{\h}_A A\restrict_{F_j} \into Z\otimes^{\h}_A B
\]
where the second arrow is the tensor product of the completely isometric embeddings $Z_i\into Z$ and $A\restrict_{F_j}\into B$. Explicitly,
\[
\phi^Z_{ij} (z_i\restrict_{F_{ij}}) = z_i\otimes 1\restrict_{F_j} \in Z_i\otimes^{\h}_A A\restrict_{F_j} \subset Z\otimes^{\h}_A B
\]
where $1\restrict_{F_j}\in \Multiplier(B)$ is as explained in Definition \ref{def:B}. Note that the tensor $z_i\otimes 1\restrict_{F_j}$ lies a priori in $Z_i\otimes^{\h}_A \Multiplier(B)$, but the nondegeneracy of $Z_i$ as an $A$-module ensures that this tensor in fact lies in the submodule $Z_i\otimes^{\h}_A B$, as we can write $z_i=z_i'a$ for some $z_i'\in Z_i$ and $a\in A$ and then identify $z_i\otimes 1\restrict_{F_j} = z_i'\otimes a 1\restrict_{F_j}\in Z_i\otimes^{\h}_A A\restrict_{F_j}$.
\end{definition}

For the last and most intricate of our technical lemmas, recall  that we identify collections $(Z_i)_{i\in I}$ of Hilbert $A\restrict_{F_i}$-modules with Hilbert $B$-modules, via the direct-sum construction (Definition \ref{def:Zi}). Thus an object $(Z,\zeta)$ of the category $\Glue(A,\mathfrak F)$ can be regarded as a Hilbert $B$-module $Z=\bigoplus_{i\in I} Z_i$, together with unitary isomorphisms $\zeta_{ij}:Z_j\restrict_{F_{ij}}\to Z_i\restrict_{F_{ij}}$ satisfying the associativity condition; and morphisms in $\Glue(A,\mathfrak F)$ can be regarded as adjointable maps of Hilbert $B$-modules compatible with the $\zeta_{ij}$s.

\begin{lemma}\label{lem:delta}
Let $(Z,\zeta)$ be an object of $\Glue(A,\mathfrak F)$. For each $z\in Z=\bigoplus_{i\in I}Z_i$ the sum
\[
\delta(z)\coloneqq \sum_{j\in I} \sum_{i\in I}  \phi^Z_{ij}\circ \zeta_{ij}( z_j\restrict_{ij})
\]
converges in $Z\otimes^{\h}_A B$. The map $\delta:Z\to Z\otimes^{\h}_A B$ is a $B$-linear complete isometry.
\end{lemma}

\begin{proof}
Let $L=\Compact_{\CMod(B)}(Z\oplus B)$ be the \emph{linking $C^*$-algebra} of the Hilbert $B$-module $Z$. See, e.g., \cite[Corollary 3.21]{RW} for details; for now it will suffice to note that $L$ is a $C^*$-algebra equipped with a completely isometric embedding $\lambda^Z:Z\to L$ and an injective $*$-homomorphism $\lambda^B:B\to L$ satisfying the relations
\begin{equation}\label{eq:lambda-relations}
\lambda^Z(zb)=\lambda^Z(z)\lambda^B(b)\quad \text{and}\quad \lambda^Z(z)^*\lambda^Z(z')=\lambda^B(\langle z|z'\rangle)
\end{equation}
for all $z,z'\in Z$ and all $b\in B$. We then embed $Z\otimes^{\h}_A B$ into the amalgamated free product $C^*$-algebra $L\ast_A L$ using the completely isometric map
\[
\Lambda: Z\otimes^{\h}_A B \xrightarrow{\lambda^Z\otimes \lambda^B} L\otimes^{\h}_A L \into L\ast_A L, \qquad \Lambda( z\otimes b) \coloneqq \lambda^Z_1(z) \lambda^B_2(b).
\]
Here the subscripts $1$ and $2$ are used to distinguish between the elements of the two embedded copies of $L$ in $L\ast_A L$. 

We shall now establish several identities related to the embedding $\Lambda$.
First we remark that
\begin{equation}\label{eq:lambda-1-2}
\lambda^B_1 ( a\restrict_{F_i} ) \lambda_2^B( 1\restrict_{F_j} ) = \lambda^B_1(1\restrict_{F_i}) \lambda^B_2(a\restrict_{F_j})
\end{equation}
for all $a\in A$. This follows from the identities $a\restrict_{F_i}=1\restrict_{F_i}\cdot a$ and $a\restrict_{F_j}=a\cdot 1\restrict_{F_j}$ in the $C^*$-algebra $B$, and from the fact that the homomorphisms $\lambda^B_1,\lambda^B_2:B\to L\ast_A L$ agree on the subalgebra $A$. 

We also have
\begin{equation}\label{eq:Lambda-phi}
\Lambda \circ \phi^Z_{ij}( z_i\restrict_{F_{ij}}) = \lambda^Z_1(z_i)\lambda^B_2(1\restrict_{F_j})
\end{equation}
for all $z_i\in Z_i$. This follows immediately upon writing $\phi^Z_{ij}(z_i\restrict_{F_{ij}})=z_i\otimes 1\restrict_{F_j}$ as in Definition \ref{def:phi}.

Using \eqref{eq:Lambda-phi}, we find that for all $i,j,k,l\in I$, all $z_i\in Z_i$ and all $z_l\in Z_l$, we have
\begin{equation}\label{eq:Lambda-star-Lambda}
\begin{aligned}
 \Lambda\left( \phi^Z_{ij}(z_i\restrict_{F_{ij}}) \right)^* \Lambda\left( \phi^Z_{lk}(z_l\restrict_{F_{lk}}) \right)  & = \lambda^B_2(1\restrict_{F_j}) \lambda^Z_1(z_i)^* \lambda^Z_1(z_l) \lambda^B_2(1\restrict_{F_k})  \\
& = \lambda^B_2(1\restrict_{F_j})\lambda^B_1( \langle z_i \, | \, z_l \rangle_B )\lambda^B_2(1\restrict_{F_k}).
\end{aligned}
\end{equation}
If $l\neq i$ then the summands $Z_i$ and $Z_l$ are orthogonal in $Z$; hence \eqref{eq:Lambda-star-Lambda} yields
\begin{equation}\label{eq:Lambda-orthogonal}
\left( \image(\Lambda\circ \phi^Z_{ij})\right)^* \left( \image(\Lambda\circ \phi^Z_{lk})\right) =0 \quad \text{if}\quad l\neq i.
\end{equation}

The final identity related to $\Lambda$ that we shall need concerns the case of $l=i$ in \eqref{eq:Lambda-star-Lambda}. Fix $j,k\in I$, $z_j\in Z_j$, and $z_k\in Z_k$. Let $a\in A$ be any element satisfying
\begin{equation}\label{eq:ajk}
a\restrict_{F_{jk}} = \left\langle \left. z_j\restrict_{F_{jk}} \, \right|\, \zeta_{jk}( z_k\restrict_{F_{jk}}) \right\rangle_{A\restrict_{F_{jk}}}.
\end{equation}
We claim that for all $i\in I$ we have
\begin{equation}\label{eq:Lambda-star-Lambda-lift}
\left( \Lambda\circ \phi^Z_{ij}\circ \zeta_{ij} (z_j\restrict_{F_{ij}})\right)^* \left(\Lambda\circ \phi_{ik}\circ \zeta_{ik}(z_k\restrict_{F_{ik}}) \right) = \lambda^B_2(1\restrict_{F_j})\lambda^B_1(a\restrict_{F_i})\lambda^B_2(1\restrict_{F_k}).
\end{equation}
To prove \eqref{eq:Lambda-star-Lambda-lift} choose elements $z_i,z_i'\in Z_i$ satisfying $\zeta_{ij}(z_j\restrict_{F_{ij}})=z_i\restrict_{F_{ij}}$ and $\zeta_{ik}(z_k\restrict_{F_{ik}})=z'_i\restrict_{F_{ik}}$. Then \eqref{eq:Lambda-star-Lambda}, with $l=i$, gives
\[
\left( \Lambda\circ \phi^Z_{ij}\circ \zeta_{ij} (z_j\restrict_{F_{ij}})\right)^* \left(\Lambda\circ \phi_{ik}\circ \zeta_{ik}(z_k\restrict_{F_{ik}}) \right) = \lambda_2^B(1\restrict_{F_j})\lambda^B_1(\langle z_i\, |\, z'_i\rangle_{A\restrict_{F_i}}) \lambda^B_2(1\restrict_{F_k}).
\]
The identity \eqref{eq:lambda-1-2} shows $\lambda_2^B(1\restrict_{F_j}) \lambda_1^B(a\restrict_{F_i}) \lambda_2^B(1\restrict_{F_k})$ depends only on $a\restrict_{F_{ijk}}$. So to  prove \eqref{eq:Lambda-star-Lambda-lift} it will suffice to show that
$
(\langle z_i \, | \, z'_i \rangle_{A\restrict_{F_i}} )\restrict_{F_{ijk}} = a\restrict_{F_{ijk}}
$,
which we do  as follows:
\[
\begin{aligned}
  \left( \langle z_i \,| \, z_i'\rangle_{A\restrict_{F_i}} \right)\restrict_{F_{ijk}} & = \left\langle \left. z_i\restrict_{F_{ijk}}\,  \right| \, z'_i\restrict_{F_{ijk}} \right\rangle_{A\restrict_{F_{ijk}}}  \\
& =  \left\langle \left. \left(\zeta_{ij}(z_j\restrict_{F_{ij}})\right)\restrict_{F_{ijk}}\,  \right| \, \left(\zeta_{ik}(z_k\restrict_{F_{ik}})\right)\restrict_{F_{ijk}}\right\rangle_{A\restrict_{F_{ijk}}} \\
& = \left\langle \left. \zeta_{ij}\restrict_{F_{ijk}}(z_j\restrict_{F_{ijk}})\,  \right| \, \zeta_{ij}\restrict_{F_{ijk}}\circ \zeta_{jk}\restrict_{F_{ijk}}(z_k\restrict_{F_{ijk}}) \right\rangle_{A\restrict_{F_{ijk}}} \\
&  = \left\langle \left. z_j\restrict_{F_{ijk}}\, \right| \, \zeta_{jk}\restrict_{F_{ijk}}(z_k\restrict_{F_{ijk}}) \right\rangle_{A\restrict_{F_{ijk}}} \\
& =\left( \left\langle \left. z_j\restrict_{F_{jk}}\, \right| \, \zeta_{jk}(z_k\restrict_{F_{jk}}) \right\rangle_{A\restrict_{F_{jk}}}\right) \restrict_{F_{ijk}} \\
& = a\restrict_{F_{ijk}}.
\end{aligned}
\]
Here the second equality comes from the definition of $z_i$ and $z'_i$; the third equality comes from the associativity property of the maps $\zeta$; the fourth equality follows from the unitarity of $\zeta_{ij}$; and the final equality holds by the definition of $a$.

Having established \eqref{eq:Lambda-star-Lambda-lift} we are now ready to address the convergence of $\delta(z)$. Since the embedding $\Lambda:Z\otimes^{\h}_A B\into L\ast_A L$ is a complete isometry, it will suffice to consider the composition
\[
\Delta \coloneqq \Lambda\circ \delta : z\mapsto \sum_{j\in I} \sum_{i\in I} \Lambda\circ \phi^Z_{ij}\circ \zeta_{ij}( z_j\restrict_{F_{ij}}).
\]

We first consider the inner sums. For each $z\in Z$, each $j\in I$, and each finite subset $S\subset I$, we let $\Delta_{j,S}(z)\in L\ast_A L$ be the partial sum
\[
\Delta_{j,S}(z) \coloneqq \sum_{i\in S} \Lambda\circ \phi^Z_{ij}\circ \zeta_{ij}(z_j\restrict_{F_{ij}}).
\]
For all $z,z'\in Z$ and all $(j,k)\in I^2$ we have
\[
\begin{aligned}
\Delta_{j,S}(z)^* \Delta_{k,S} (z') & = \sum_{(i,l)\in S^2} \left( \Lambda\circ \phi^Z_{ij}\circ \zeta_{ij} (z_j\restrict_{F_{ij}})\right)^* \left(\Lambda\circ \phi^Z_{lk}\circ \zeta_{lk}(z'_k\restrict_{F_{lk}}) \right) \\
& = \sum_{i\in S} \left( \Lambda\circ \phi^Z_{ij}\circ \zeta_{ij} (z_j\restrict_{F_{ij}})\right)^* \left(\Lambda\circ \phi^Z_{ik}\circ \zeta_{ik}(z'_k\restrict_{F_{ik}}) \right) 
\end{aligned}
\]
since \eqref{eq:Lambda-orthogonal} ensures that the terms with $l\neq i$ are zero. Letting $a_{jk}\in A$ be a lift of $\langle z_j\restrict_{F_{jk}}\, | \, \zeta_{jk}(z'_k\restrict_{F_{jk}}) \rangle_{A\restrict_{F_{jk}}}$, an application of \eqref{eq:Lambda-star-Lambda-lift}  gives
\[
\begin{aligned}
\Delta_{j,S}(z)^* \Delta_{k,S} (z')  & = \sum_{i\in S} \lambda^B_2(1\restrict_{F_j})\lambda^B_1(a_{jk}\restrict_{F_i}) \lambda^B_2(1\restrict_{F_j}) \\ 
& = \lambda^B_2(1\restrict_{F_j}) \lambda^B_1\left( \sum_{i\in S} a_{jk}\restrict_{F_i}\right) \lambda^B_2(1\restrict_{F_j}).
\end{aligned}
\]
The net $S\mapsto  \sum_{i\in S} a_{jk}\restrict_{F_i}$ converges in $B$ to the element $a_{jk}\in A$, and so the net $S\mapsto \Delta_{j,S}(z)^*\Delta_{k,S}(z')$ converges in $L\ast_A L$ to the element 
\[
\begin{aligned}
\lambda^B_2(1\restrict_{F_j})\lambda^B_1(a_{jk})\lambda^B_2(1\restrict_{F_k}) & = \lambda^B_2(1\restrict_{F_j})\lambda^B_2(a_{jk})\lambda^B_2(1\restrict_{F_k}) \\
& = \begin{cases} \lambda^B_2( a_{jj}\restrict_{F_j}) = \lambda^B_2(\langle z_j \, |\, z'_j\rangle_{A\restrict_{F_j}}) & \text{if $j=k$} \\
0 & \text{if $j\neq k$}\end{cases}
\end{aligned}
\]
because if $j\neq k$ the central idempotents $1\restrict_{F_j}, 1\restrict_{F_k}\in \Multiplier(B)$ are orthogonal. 

Putting $k=j$ and $z'=z$ shows that the sum 
\[
\Delta_j(z)\coloneqq \sum_{i\in I} \Lambda\circ \phi_{ij}\circ \zeta_{ij}(z_j\restrict_{ij})
\]
converges in $L\ast_A L$. Moreover, we have
\begin{equation}\label{eq:Deltaj-star-Deltaj}
\Delta_j(z)^*\Delta_k(z')=\begin{cases} \lambda^B_2(\langle z_j\, |\, z'_j\rangle_{A\restrict_{F_j}}) & \text{if $j=k$} \\ 0& \text{if $j\neq k$.}\end{cases}
\end{equation}
for all $(j,k)\in I^2$ and all $z,z'\in Z$.

Now for each finite subset $S\subseteq I$ and for all $z,z'\in Z$, \eqref{eq:Deltaj-star-Deltaj} implies that
\[
\left( \sum_{j\in S}\Delta_j(z)\right)^* \left( \sum_{k\in S}\Delta_k(z)\right) = \lambda^B_2\left( \sum_{j\in S} \langle z_j\, |\, z'_j\rangle_{A\restrict_{F_j}}\right).
\]
Since the net $S\mapsto \sum_{j\in S} \langle z_j\, |\, z'_j\rangle_{A\restrict_{F_j}}$ converges in $B$ to $\langle z\, |\, z'\rangle_B$, we finally conclude that the sum $\Delta(z)=\sum_{j\in I}\Delta_j(z)$ converges in $L\ast_A L$, and satisfies
\begin{equation}\label{eq:Delta-star-Delta}
\Delta(z)^* \Delta(z')  = \lambda^B_2\left(\langle z\, |\, z'\rangle_B \right)
\end{equation}
for all $z,z'\in Z$.

Now that we know that $\delta$ is well-defined, it is easy to see that it is $B$-linear, and we are thus left to show that $\delta$ is a complete isometry; since $\Lambda$ is a complete isometry, it will suffice to show that $\Delta=\Lambda\circ\delta$ is a complete isometry. To do this, note that the equality \eqref{eq:Delta-star-Delta} implies that the canonical $L\ast_A L$-valued inner product on the $C^*$-algebra $L\ast_A L$, restricted to the image $\Delta(Z)$, takes values in the subalgebra $\lambda^B_2(B)\cong B$. The same formula shows that if we use this $\lambda^B_2(B)$-valued inner product to consider $\Delta(Z)$ as a Hilbert $B$-module, then the map $\Delta:Z\to \Delta(Z)$ is an isometric $B$-linear isomorphism, hence a unitary isomorphism of Hilbert $B$-modules, and hence a complete isometry.
\end{proof}

\section{Gluing Hilbert modules}\label{sec:gluing}

\subsection{Gluing elements}
We continue to assume Hypotheses \ref{hyp}.
Let $X$ be a right Hilbert $A$-module. For each $i\in I$ we have a Hilbert $A\restrict_{F_i}$-module $X\restrict_{F_i}$, and the direct sum $\bigoplus_{i\in I} X\restrict_{F_i}\cong X\otimes^{\h}_A B$ is a Hilbert module over $B=\bigoplus_{i\in I} A\restrict_{F_i}$; see Lemma \ref{lem:X-tensor-B}. The quotient maps $X\onto X\restrict_{F_i}$ assemble into a map $X\to \bigoplus_{i\in I} X\restrict_{F_i}$, and our next goal is to compute the image of this map. To do this we shall use the following notation:

\begin{definition}\label{def:eta}
For each operator $A$-module $X$ we define a map $\eta^X:X\to X\otimes^{\h}_A B$ by $\eta^X(x)\coloneqq x\otimes 1$. Here the tensor $x\otimes 1$ ostensibly lies in $X\otimes^{\h}_A \Multiplier(B)$;  but it in fact lies in the submodule $X\otimes^{\h}_A B$, by the same reasoning as was applied to $z_i\otimes 1\restrict_{F_j}$ in Definition \ref{def:phi}.
\end{definition}

 The map $\eta^X$ fits into a commuting diagram
\begin{equation}\label{eq:eta-diag}
\xymatrix{ 
X \ar[rr]^-{\eta^X} & & X\otimes^{\h}_A B \\
& X\otimes^{\h}_A A\ar[ul]_-{\cong}^-{x\otimes a\mapsto xa} \ar[ur]_-{\id_X\otimes \eta}
}
\end{equation}
where $\eta:A\to B$ is the embedding of Lemma \ref{lem:A-in-B}. The two diagonal arrows in the above diagram are complete isometries, by  Theorem \ref{thm:Haagerup}  parts (a) and (c); hence $\eta^X$ is also a complete isometry.

\begin{lemma}\label{lem:image-eta}
For each operator $A$-module $X$ we have 
\[
\image \eta^X=\ker \left( \eta^X\otimes \id_B - \id_X\otimes \eta^B : X\otimes^{\h}_A B\to X\otimes^{\h}_A B\otimes^{\h}_A B\right).
\]
\end{lemma}

This is an instance of \cite[Proposition 4.7(a)]{Crisp-descent}. In accordance with our aim to make this paper readable independently of the full machinery and notation of \cite{Crisp-descent} we shall briefly recall the proof.

\begin{proof}
The containment of the left-hand side in the right-hand side is easily verified.
For the reverse containment,  let $q:B\to B/A$ be the quotient mapping of operator $A$-modules. It is easy to see that $\eta^{B/A}\circ q = (q\otimes \id_B)\circ \eta^B$ as maps $B\to (B/A)\otimes^{\h}_A B$. Theorem \ref{thm:Haagerup} part (c), together with the commuting diagram \eqref{eq:eta-diag}, implies that $\image \eta^X=\ker(\id_X\otimes q)$. A short computation shows that if $t$ lies in $\ker(\eta^X\otimes \id_B - \id_X\otimes \eta^B)$ then we have
\[
(\id_X\otimes \eta^{B/A})\circ (\id_X\otimes q) (t) = (\id_X \otimes q\otimes \id_B)\circ (\eta^X\otimes \id_B) (t),
\]
and the right-hand side is zero because $\image \eta^X = \ker(\id_X\otimes q)$. Now $\eta^{B/A}$ is a complete isometry, so $\id_X\otimes \eta^{B/A}$ is likewise, and therefore the equation in the last display implies that $t\in \ker(\id_X\otimes q)=\image \eta^X$.
\end{proof}

\begin{lemma}\label{lem:gluing-elements}
Let $X$ be a Hilbert $A$-module. The formula
\[
\Phi^X: X\to \bigoplus_{i\in I} X\restrict_{F_i},\qquad \Phi^X(x)\coloneqq \left( x\restrict_{F_i}\right)_{i\in I}
\]
gives a well-defined $A$-linear complete isometry, with
\[
\image\Phi^X = \left\{ \left. (x_i)_{i\in I} \in \bigoplus_{i\in I} X\restrict_{F_i} \, \right| \, x_i\restrict_{F_{ij}} =  x_j\restrict_{F_{ij}}\text{ for all }(i,j)\in I^2 \right\}.
\]
\end{lemma}

\begin{example}\label{example:eta-A}
Taking $X=A$ in Lemma \ref{lem:gluing-elements} yields
\[
A = \left\{ (b_i)_{i\in I}\in B\ \left|\ b_i\restrict_{F_{ij}} = b_j\restrict_{F_{ij}}\text{ for all }(i,j)\in I^2 \right.\right\}.
\]
\end{example}

\begin{proof}[Proof of Lemma \ref{lem:gluing-elements}]
Firstly, to see that the given formula for $\Phi^X(x)$ actually defines an element of the $C_0$-direct sum, we compute
\[
\begin{aligned}
\left\| x\restrict_{F_i} \right\|^2 & = \left\| \left\langle x\restrict_{F_i} \, \big|\, x\restrict_{F_i} \right\rangle_{A\restrict_{F_i}} \right\|  = \left\| (\langle x\, |\, x\rangle_A )\restrict_{F_i} \right\| 
\end{aligned}
\]
and note that the right-hand norm vanishes as $i\to\infty$, as shown in Lemma \ref{lem:A-in-B}. The $A$-linearity of $\Phi^X$ is easily checked.

Now $\Phi^X$ fits in to a commuting diagram 
\[
\xymatrix@C=50pt@R=10pt{
& \bigoplus_{i\in I} X\restrict_{F_i}  \\ 
X \ar[ur]^-{\Phi^X} \ar[dr]_-{\eta^X} & \\
& X\otimes^{\h}_A B \ar[uu]_-{\Psi^X}
}
\]
where $\Psi^X$ is the isomorphism from Lemma \ref{lem:X-tensor-B}, and where $\eta^X$ is as in Definition \ref{def:eta}. Since $\Psi^X$ and $\eta^X$ are compete isometries, $\Phi^X$ is likewise. 

To compute the image of $\Phi^X$ we first note that Lemma \ref{lem:mu} ensures that the maps 
\[
\mu_{ij}=\mu^{X\otimes^{\h}_A B}_{ij}: X\otimes^{\h}_A B\otimes^{\h}_A B\to X\restrict_{ij},\qquad  x\otimes b \otimes b' \mapsto  (x \restrict_{F_i} \cdot b_i)\restrict_{F_{ij}} \cdot b'_j\restrict_{F_{ij}} 
\]
separate points, as $(i,j)$ ranges over $I^2$. In view of Lemma \ref{lem:image-eta}, we thus have
\[
\image \eta^X = \bigcap_{(i,j)\in I^2} \ker\left( \mu_{ij}\circ (\eta^X\otimes \id_B - \id_X\otimes \eta^B)\right).
\]
Now we compute, for a fixed $(i,j)\in I^2$, 
\[
\begin{aligned}
 \mu_{ij}\circ (\eta^X\otimes \id_B - \id_X \otimes \eta^B) (x\otimes b) & = \mu_{ij} \left( x\otimes 1 \otimes b- x\otimes b \otimes 1 \right) \\
& = x\restrict_{F_{ij}}\cdot b_j\restrict_{F_{ij}} - (x\restrict_{F_i}\cdot b_i)\restrict_{F_{ij}}.
\end{aligned}
\]
Writing $
\theta_{ij} \left( (x_i)_{i\in I}\right) \coloneqq x_j\restrict_{ij} - x_i\restrict_{ij}
$
we find that
\[
\mu_{ij}\circ (\eta^X\otimes \id_B - \id_X\otimes \eta_B) = \theta_{ij}\circ \Psi^X 
\]
and consequently
\[
\begin{aligned}
\image \Phi^X  = \Psi^X (\image \eta^X) &= \Psi^X \left( \bigcap_{(i,j)\in I^2} \ker \left( \mu_{ij}\circ (\eta^X\otimes \id_B - \id_X\otimes \eta^B)\right) \right)\\
& = \bigcap_{(i,j)\in I^2} \ker \theta_{ij}
\end{aligned}
\]
which is what we wanted to prove.
\end{proof}

\subsection{The gluing functor}\label{subsec:G}

We now define a gluing functor $\functor{G}:\Glue(A,\mathfrak F)\to \CMod(A)$ that will be inverse to the pulling-apart functor $\functor{P}$. 

Given a pair $(Z,\zeta)\in \Glue(A,\mathfrak F)$ we consider the following closed subspace of the Hilbert $B$-module $Z=\bigoplus_{i\in I} Z_i$:
\begin{equation}\label{eq:G-def}
\functor{G}(Z,\zeta)\coloneqq \left\{ \left. (z_i)_{i\in I} \in \bigoplus_{i\in I}Z_i\ \right|\ z_i\restrict_{F_{ij}} = \zeta_{ij}(z_j\restrict_{F_{ij}})\text{ for all }(i,j)\in I^2 \right\}.
\end{equation}

\begin{lemma}\label{lem:GZ-over-A}
The subspace $\functor{G}(Z,\zeta)$ of $Z$ is an $A$-submodule, and for all $z,z'\in \functor{G}(Z,\zeta)$ we have $\langle z \, | \, z'\rangle_B\in A$. Thus $\functor{G}(Z,\zeta)$ is a  Hilbert $A$-module.
\end{lemma}

\begin{proof}
The fact that $\functor{G}(Z,\zeta)$ is stable under right multiplication follows easily from $A\restrict_{F_{ij}}$-linearity of $\zeta_{ij}$.
For the assertion about the inner product, let $z$ and $z'$ be two elements of $\functor{G}(Z,\zeta)$. To show that $\langle z \, | \,z'\rangle_B$ lies in $A$ it will suffice, by Example \ref{example:eta-A}, to prove that for each $(i,j)\in I^2$ we have
$
( \langle z_i \, |\, z'_i \rangle_{A\restrict_{F_i}})\restrict_{F_{ij}} = ( \langle z_j \, | \, z'_j \rangle_{A\restrict_{F_j}})\restrict_{F_{ij}}.
$
This equality is established as follows:
\begin{equation*}\label{eq:GZ-over-A-pf2}
\begin{aligned}
\left( \langle z_i \,| \, z'_i \rangle_{A\restrict_{F_i}}\right)\restrict_{F_{ij}} & = \left\langle \left. z_i\restrict_{F_{ij}} \, \right|\, z'_i \restrict_{F_{ij}}  \right\rangle_{A\restrict_{F_{ij}}}  =\left\langle \left. \zeta_{ij}( z_j\restrict_{F_{ij}})\, \right|\, \zeta_{ij}(z'_j\restrict_{F_{ij}}) \right\rangle_{A\restrict_{F_{ij}}} \\ & = \left\langle \left. z_j \restrict_{F_{ij}}\, \right|\, z'_j\restrict_{F_{ij}} \right\rangle_{A\restrict_{F_{ij}}} 
 =\left( \langle z_j \, | \, z'_j \rangle_{A_{F_j}}\right)\restrict_{F_{ij}}
\end{aligned}
\end{equation*}
where the second equality holds thanks to our assumption that  $z$ and $z'$  lie in $\functor{G}(Z,\zeta)$, while the third equality holds because $\zeta_{ij}$ is a unitary isomorphism.
\end{proof}

\begin{lemma}\label{lem:GZ-morphisms}
Let $\alpha=(\alpha_i)_{i\in I}:(Z,\zeta)\to (W,\omega)$ be a morphism in $\Glue(A,\mathfrak F)$. The map $Z\to W$, $(z_i)_{i\in I}\mapsto (\alpha_i(z_i))_{i\in I}$ restricts to an adjointable map  of Hilbert $A$-modules, $\functor{G}\alpha:\functor{G}(Z,\zeta)\to \functor{G}(W,\omega)$.
\end{lemma}

\begin{proof}
First note, as we did in Definition \ref{def:Zi}, that $(z_i)_{i\in I}\mapsto (\alpha_i(z_i))_{i\in I}$ is a well-defined adjointable map of Hilbert $B$-modules $Z\to W$. Let us show that this map sends $\functor{G}(Z,\zeta)$ into $\functor{G}(W,\omega)$. Given $z=(z_i)_{i\in I}\in \functor{G}(Z,\zeta)$, we compute for all $(i,j)\in I^2$:
\begin{multline*}
 (\alpha_i(z_i))\restrict_{F_{ij}} = \alpha_i\restrict_{F_{ij}} \left( z_i\restrict_{F_{ij}}\right) = \alpha_i\restrict_{F_{ij}} \circ \zeta_{ij}\left( z_j\restrict_{F_{ij}}\right) = \omega_{ij}\circ \alpha_j\restrict_{F_{ij}} \left( z_j\restrict_{F_{ij}}\right) \\
 = \omega_{ij}\left( (\alpha_j(z_j)\restrict_{F_{ij}}\right).
\end{multline*}
Thus $\alpha$ induces a map $\functor{G}\alpha:\functor{G}(Z,\zeta)\to \functor{G}(W,\omega)$, as required.

The same argument applied to $\alpha^*=(\alpha_i^*)_{i\in I}:(W,\omega)\to (Z,\zeta)$ shows that $\functor{G}(\alpha^*)$ sends $\functor{G}(W,\omega)$ into $\functor{G}(Z,\zeta)$, and the fact that each $\alpha_i^*$ is adjoint to $\alpha_i$ then ensures that $\functor{G}(\alpha^*)$ is adjoint to $\functor{G}\alpha$. 
\end{proof}

\begin{definition}\label{def:G}
Define a $*$-functor $\functor{G}:\Glue(A,\mathfrak F) \to \CMod(A)$ on objects by defining $\functor{G}(Z,\zeta)$ as in \eqref{eq:G-def}, and on morphisms by defining $\functor{G}\alpha$ as in Lemma \ref{lem:GZ-morphisms}.
\end{definition}

We are going to prove Theorem \ref{thm:main} by showing that $\functor{G}$ is an inverse to $\functor{P}$. The proof occupies Sections \ref{subsec:pf-begin}--\ref{subsec:pf-end}

\subsection{Proof that $\functor{G}\functor{P}\cong \id$}\label{subsec:pf-begin}

Let $X$ be a Hilbert $A$-module. Then $\functor{G}\functor{P}X$ is an $A$-submodule of the Hilbert $B$-module $\bigoplus_{i\in I} X\restrict_i$; specifically, 
\[
\functor{G}\functor{P}X = \left\{ \left. (x_i)_{i\in I} \in \bigoplus_{i\in I} X\restrict_i \ \right|\ x_i\restrict_{ij}=x_j\restrict_{ij}\text{ for all }(i,j)\in I^2 \right\}
\]
where we are suppressing the canonical isomorphisms $\kappa^X_{ij}:X\restrict_j\restrict_{ij} \to X\restrict_i \restrict_{ij}$.

Lemma \ref{lem:gluing-elements} shows that the map $\Phi^X: X\to \functor{G}\functor{P}X$ is an isometric $A$-linear isomorphism, and since both $X$ and $\functor{G}\functor{P}X$ are Hilbert $A$-modules this ensures that $\Phi^X$ is a unitary isomorphism. The fact that $\Phi^X$ is natural in $X$ is clear from its definition, and so  $\Phi:\id_{\CMod(A)}\to \functor{G}\functor{P}$ is a unitary natural isomorphism.\hfill\qed

\subsection{Proof that $\functor{P}\functor{G}\cong \id$, begun}

Let $(Z,\zeta)$ be an object in $\Glue(A,\mathfrak F)$, and consider the object $\functor{P}\functor{G}(Z,\zeta)$ of $\Glue(A,\mathfrak F)$. For each $i\in I$ we have 
\[
\functor{P}\functor{G}(Z,\zeta)_i = \functor{G}(Z,\zeta)\restrict_i
\]
with the gluing isomorphisms 
\[
\kappa^{\functor{G}(Z,\zeta)}_{ij}: \functor{G}(Z,\zeta)\restrict_j\restrict_{ij}\xrightarrow{\cong} \functor{G}(Z,\zeta)\restrict_i\restrict_{ij}
\]
being the canonical isomorphisms as in Definition \ref{def:P}. Lemma \ref{lem:X-tensor-B} gives a natural isomorphism 
\[
\bigoplus_{i\in I} \functor{P}\functor{G}(Z,\zeta)_i\cong \functor{G}(Z,\zeta)\otimes^{\h}_A B.
\]
(Recall from Theorem \ref{thm:Haagerup} part (b) that the Hilbert-module tensor product is completely isometrically isomorphic to the Haagerup tensor product.) 

Since $\functor{G}(Z,\zeta)$ is by definition a Hilbert $A$-submodule of the Hilbert $B$-module $Z=\bigoplus_{i\in I} Z_i$,  the Hilbert $B$-module $\functor{G}(Z,\zeta)\otimes^{\h}_A B$ embeds completely isometrically into  $Z \otimes^{\h}_A B$ (Theorem \ref{thm:Haagerup} part (c)). This latter tensor product is an operator $B$-module, but it is not a Hilbert $B$-module, because $Z$ is not a Hilbert $A$-module. 

Since $Z$ is an operator $B$-module we have a completely contractive multiplication map $\epsilon^Z:Z\otimes^{\h}_A B\xrightarrow{z\otimes b\mapsto zb} Z$. We consider the restriction of $\epsilon^Z$ to the Hilbert $B$-module $\functor{G}(Z,\zeta)\otimes^{\h}_A B$.

\begin{lemma}\label{lem:epsilon}
\begin{enumerate}[\rm(a)]
\item The map $\epsilon^Z: \functor{G}(Z,\zeta)\otimes^{\h}_A B \to Z$ satisfies 
\[
\langle \epsilon^Z(x)\, |\, \epsilon^Z(y) \rangle_B = \langle x\, |\, y\rangle_B
\]
 for all $x,y\in \functor{G}(Z,\zeta)\otimes^{\h}_A B$.
\item For each $i\in I$ let $\epsilon^Z_i : \functor{G}(Z,\zeta)\restrict_{F_i} \to Z_i$ be the composition
\[
 \functor{G}(Z,\zeta)\restrict_{F_i} \xrightarrow{\left(\pi^{\functor{P}\functor{G}(Z,\zeta)}_i\right)^*} \bigoplus_{j\in I} \functor{G}(Z,\zeta)\restrict_{F_j} \xrightarrow{\left(\Psi^{\functor{G}(Z,\zeta)}\right)^{-1}} \functor{G}(Z,\zeta)\otimes^{\h}_A B \xrightarrow{\epsilon^Z} Z \xrightarrow{\pi^Z_i} Z_i.
\]
For all $(i,j)\in I^2$ we have
\[
\epsilon^Z_i \restrict_{F_{ij}}\circ \kappa_{ij}^{\functor{G}(Z,\zeta)} = \zeta_{ij}\circ \epsilon^Z_j\restrict_{F_{ij}}
\]
as maps $\functor{G}(Z,\zeta)\restrict_{F_j} \restrict_{F_{ij}} \to Z_i\restrict_{F_{ij}}$.
\item If $\epsilon^Z$ maps $\functor{G}(Z,\zeta)\otimes^{\h}_A B$ surjectively onto $Z$, for every object $(Z,\zeta)$ of $\Glue(A,\mathfrak F)$, then the functor $\functor{P}\functor{G}$ is unitarily naturally isomorphic to $\id_{\Glue(A,\mathfrak F)}$.
\end{enumerate}
\end{lemma}

\begin{proof}
Part (a) is easily checked: for $z,z'\in\functor{G}(Z,\zeta)$ and  $b,b'\in B$ we have
\[
\langle \epsilon^Z(z\otimes b)\, |\, \epsilon^Z(z'\otimes b') \rangle_B = \langle zb\, |\, z'b'\rangle_B = b^*\langle z\, |\, z' \rangle_B b' = \langle z\otimes b\, |\, z'\otimes b'\rangle_B,
\]
where the final equality makes sense because, as shown in Lemma \ref{lem:GZ-over-A}, the inner product $\langle z\, |\, z'\rangle_B$ actually lies in the subalgebra $A$.

For part (b): the map
$
\epsilon^Z_i \restrict_{F_{ij}}\circ \kappa^{\functor{G}(Z,\zeta)}_{ij} :  \functor{G}(Z,\zeta)\restrict_{F_j} \restrict_{F_{ij}} \to Z_i \restrict_{F_{ij}}
$
is given, for $z\in \functor{G}(Z,\zeta)$, by
\[
z\restrict_{F_j} \restrict_{F_{ij}} \xmapsto{\kappa_{ij}} z\restrict_{F_i}\restrict_{F_{ij}} \xmapsto{\epsilon_i\restrict_{F_{ij}}} z_i \restrict_{F_{ij}}
\]
while the map
$
\zeta_{ij}\circ \epsilon^Z_j\restrict_{F_{ij}}  : \functor{G}(Z,\zeta)\restrict_{F_j} \restrict_{F_{ij}} \to Z_i\restrict_{F_{ij}}
$
is given by
\[
z\restrict_{F_j}\restrict_{F_{ij}} \xmapsto{ \epsilon_j\restrict_{F_{ij}}} z_j\restrict_{F_{ij}}\xmapsto{\zeta_{ij}} \zeta_{ij}(z_j\restrict_{F_{ij}}).
\]
The fact that $z$ lies in $\functor{G}(Z,\zeta)$ means that $z_i\restrict_{F_{ij}} = \zeta_{ij}(z_j\restrict_{F_{ij}})$ for all $(i,j)$, and this proves  part (b).

For part (c), if $\epsilon^Z:\functor{G}(Z,\zeta)\otimes^{\h}_A B\to Z$ is surjective then part (a) ensures that this map is a unitary isomorphism of Hilbert $B$-modules, while part (b) shows that this map is a morphism $\functor{P}\functor{G}(Z,\zeta)\to (Z,\zeta)$ in $\Glue(A,\mathfrak F)$. The multiplication map $\epsilon^Z:Z\otimes^{\h}_A B\to Z$ is natural in $Z$ with respect to all completely bounded $B$-module maps, so it is certainly natural with respect to all of the morphisms in $\Glue(A,\mathfrak F)$; thus the unitary isomorphisms $\epsilon^Z$ combine to give a natural unitary isomorphism $\functor{P}\functor{G}\to \id_{\Glue}$.
\end{proof}

\subsection{Algebraic properties of the map $\delta$} 
Fix an object $(Z,\zeta)$ of $\Glue(A,\mathfrak F)$. Lemma \ref{lem:epsilon} shows that in order to complete the proof of Theorem \ref{thm:main} it will suffice to show that the multiplication map $\epsilon^Z:\functor{G}(Z,\zeta)\otimes^{\h}_A B\to Z$ is surjective. We do this by proving that the map $\delta:Z\to Z\otimes^{\h}_A B$ defined in Lemma \ref{lem:delta} has image contained in $\functor{G}(Z,\zeta)\otimes^{\h}_A B$, and satisfies $\epsilon^Z\circ \delta = \id_Z$. The second part is easy (see part (a) of Lemma \ref{lem:delta-algebra}, below), but the first is more challenging, and it is at this point that we use the weak expectation $\iota:B\to A^{\dual\dual}$ from Lemma \ref{lem:A-in-B}, importing the main technical argument from \cite{Crisp-descent}; see Lemma \ref{lem:kernels}.

\begin{lemma}\label{lem:delta-algebra}
Let $(Z,\zeta)$ be an object of $\Glue(A,\mathfrak F)$,  let $\delta:Z\to Z\otimes^{\h}_A B$ be the map defined in Lemma \ref{lem:delta}, and let $\eta^Z:Z\to Z\otimes^{\h}_A B$ be as in Definition \ref{def:eta}. We have:
\begin{enumerate}[\rm(a)]
\item $\epsilon^Z\circ \delta = \id_Z$.
\item $(\delta\otimes\id_B)\circ \delta = (\eta_Z\otimes \id_B)\circ \delta$ as maps $Z\to Z\otimes^{\h}_A B\otimes^{\h}_A B$.
\item $\functor{G}(Z,\zeta) = \ker(\eta^Z-\delta)$.
\end{enumerate}
\end{lemma}

\begin{proof}
For part (a): let $\phi^Z_{ij}:Z_i\restrict_{F_{ij}}\into Z\otimes^{\h}_A B$ be as in Definition \ref{def:phi}. For each $(i,j)\in I^2$ the map $\epsilon^Z\circ \phi^Z_{ij} : Z_i\restrict_{F_{ij}} \to Z$ is given by
\[
\epsilon^Z\circ \phi^Z_{ij} (z_i\restrict_{ij}) = \epsilon^Z (z_i\otimes 1\restrict_{F_j}) = \begin{cases} z_i & \text{if $i=j$} \\ 0 & \text{if $i\neq j$,}\end{cases}
\]
and so for each $z\in Z$ we have
\[
\epsilon^Z\circ \delta(z) = \sum_{j\in I}\sum_{i\in I} \epsilon^Z\circ \phi^Z_{ij}\circ \zeta_{ij}(z_j\restrict_{F_{ij}}) = \sum_{i\in I} \zeta_{ii}(z_i\restrict_{F_{ii}}) = \sum_{i\in I} z_i =z.
\]

For part (b): Lemma \ref{lem:mu} implies that the maps 
\[
\mu^Z_{ijk}: Z\otimes^{\h}_A B\otimes^{\h}_A B \to Z_i\restrict_{F_{ijk}},\qquad z\otimes b\otimes b'\mapsto z_i\restrict_{F_{ijk}}\cdot b_j\restrict_{F_{ijk}}\cdot b'_k\restrict_{F_{ijk}}
\]
separate points, as $(i,j,k)$ ranges over $I^3$. For each $(i,j,k)$ the map $\mu^Z_{ijk}$ factors through the projection 
\[
\pi^Z_i\otimes \pi^B_j\otimes \pi^B_k : Z\otimes^{\h}_A B\otimes^{\h}_A B \to Z_i\otimes^{\h}_A A\restrict_{F_j} \otimes^{\h}_A A\restrict_{F_k} \subset Z\otimes^{\h}_A B \otimes^{\h}_A B
\]
so that $\mu^Z_{ijk} = \mu^Z_{ijk}\circ (\pi^Z_i\otimes\pi^B_j\otimes \pi^B_k)$. Now, for each $(i,j)\in I^2$ define 
\[
\delta_{ij}:Z\to Z\otimes^{\h}_A B,\qquad \delta_{ij}(z)\coloneqq \phi_{ij}\circ \zeta_{ij}(z_j\restrict_{F_{ij}}).
\] 
Then we have $\delta_{ij} = (\pi^Z_i\otimes\pi^B_j)\circ \delta_{ij} \circ \pi^Z_j$, from which it follows that
\[
\mu_{ijk}^Z \circ (\delta_{pq}\otimes \id_B) \circ \delta_{rs} = 0 \quad \text{unless}\quad p=i,\ q=r=j,\ s=k.
\]
Since $\delta = \sum_{j\in I}\sum_{i\in I}\delta_{ij}$, it follows that
\[
\mu_{ijk}^Z\circ (\delta\otimes \id_B)\circ \delta =  \mu_{ijk}^Z \circ (\delta_{ij}\otimes \id_B)\circ \delta_{jk}.
\]

Fix $z\in Z$, choose $w_j\in Z_j$ with $w_j\restrict_{F_{jk}}=\zeta_{jk}(z_k\restrict_{F_{jk}})$, and then choose $w_i\in Z_i$ with $w_i\restrict_{F_{ij}}=\zeta_{ij}(w_j\restrict_{F_{ij}})$. We then have $\delta_{jk}(z) = w_j\otimes 1\restrict_{F_k}$ and $\delta_{ij}(w_j) = w_i\otimes 1\restrict_{F_j}$, and so
\begin{equation}\label{eq:delta-b-pf1}
\begin{aligned}
&\mu_{ijk}^Z\circ (\delta \otimes \id_B)\circ \delta (z)  = \mu_{ijk}^Z\circ (\delta_{ij}\otimes \id_B) \circ \delta_{jk}(z) \\
& = \mu_{ijk}^Z(w_i\otimes 1\restrict_{F_j} \otimes 1\restrict_{F_k})  = w_i\restrict_{F_{ijk}} = \zeta_{ij}(w_j\restrict_{F_{ij}})\restrict_{F_{ijk}}  = \zeta_{ij}\restrict_{F_{ijk}} ( w_j\restrict_{F_{ijk}}) \\ 
& = \zeta_{ij}\restrict_{F_{ijk}} \circ \zeta_{jk}\restrict_{F_{ijk}} (z_k\restrict_{F_{ijk}})  = \zeta_{ik}\restrict_{F_{ijk}}(z_k\restrict_{F_{ijk}}).
\end{aligned}
\end{equation}
On the other hand, we have $(\pi^Z_i\otimes \id_B)\circ \eta^Z = \eta^Z\circ \pi^Z_i$, and so writing $\mu_{ijk}^Z = \mu_{ijk}^Z\circ (\pi^Z_i\otimes \id_B \otimes \pi^B_k)$ gives
\[
\begin{aligned}
\mu_{ijk}^Z\circ (\eta^Z\otimes \id_B)\circ \delta & = \mu_{ijk}^Z \circ (\eta^Z\otimes \id_B) \circ (\pi^Z_i\otimes \pi^B_k)\circ \delta \\ 
& = \mu_{ijk}^Z\circ (\eta^Z\otimes \id_B)\circ \delta_{ik}.
\end{aligned}
\]
Given $z\in Z$ we choose $x_i\in Z_i$ with $x_i\restrict_{F_{ik}}=\zeta_{ik}(z_k\restrict_{F_{ik}})$, so that $\delta_{ik}(z) = x_i\otimes 1\restrict_{F_k}$. Now 
\[
\begin{aligned}
\mu_{ijk}^Z\circ (\eta^Z\otimes \id_B)\circ \delta (z) & = \mu_{ijk}^Z\circ (\eta^Z\otimes \id_B) (x_i\otimes 1\restrict_{F_k}) = \mu_{ijk}^Z(x_i\otimes 1\otimes 1\restrict_{F_k}) \\
& = x_i\restrict_{F_{ijk}} = \zeta_{ik}\restrict_{F_{ijk}}(z_k\restrict_{F_{ijk}})
\end{aligned}
\]
which we showed in \eqref{eq:delta-b-pf1} to equal $\mu_{ijk}^Z\circ (\delta\otimes \id_B)\circ \delta(z)$. Since the maps $\mu_{ijk}^Z$ separate points, this concludes the proof of part (b).

For part (c): consider the maps 
\[
\mu^Z_{ij}:Z\otimes^{\h}_A B\to Z_i\restrict_{ij},\qquad z\otimes b\mapsto z_i\restrict_{F_{ij}}\cdot b_j\restrict_{F_{ij}}.
\]
A computation similar to the one used to prove part (b) shows that for each $z\in Z$ we have
\[
\mu^Z_{ij}\circ (\eta^Z-\delta) (z) = z_i\restrict_{F_{ij}} - \zeta_{ij}(z_j\restrict_{ij}).
\]
Consulting the definition of $\functor{G}(Z,\zeta)$ then shows that
\[
\functor{G}(Z,\zeta) = \bigcap_{(i,j)\in I^2} \ker \left( \mu^Z_{ij}\circ (\eta^Z-\delta)\right) = \ker(\eta^Z-\delta)
\]
where the second equality holds because the maps $\mu_{ij}^Z$ separate the points of $Z\otimes^{\h}_A B$ (Lemma \ref{lem:mu} once again.)
\end{proof}

\begin{lemma}\label{lem:kernels}
We have $\Glue(Z,\zeta)\otimes^{\h}_A B = \ker \left( (\eta^Z-\delta)\otimes \id_B\right)\subset Z\otimes^{\h}_A B$.
\end{lemma}

This is an instance of \cite[Corollary 4.10]{Crisp-descent}.
We shall present here an outline of the argument, referring the reader to \cite[Lemma 4.8, Lemma 4.9, Corollary 4.10]{Crisp-descent} for the details.

\begin{proof}
Let $\gamma=\eta^Z-\delta$. Lemma \ref{lem:delta-algebra} part (c) says that $\functor{G}(Z,\zeta) = \ker \gamma$. If $\gamma$ is conjugate, via a completely bounded isomorphism, to a quotient mapping of operator spaces, then Theorem \ref{thm:Haagerup} part (c) implies that $(\ker\gamma)\otimes^{\h}_A B = \ker \left( \gamma\otimes \id_B\right)$, which is what we are trying to prove.

The basic duality theory of operator spaces (cf.~\cite[Section 1.4]{BLM}) implies  that in order to show that $\gamma$ is conjugate to a quotient mapping, it is enough to show that the dual map $\gamma^{\dual} : (Z\otimes^{\h}_A B)^{\dual}\to Z^\dual$ is conjugate to a quotient mapping. This property in turn follows from the existence of a pseudoinverse to $\gamma^\dual$, i.e., a completely bounded map $\sigma:Z^\dual\to (Z\otimes^{\h}_A B)^\dual$ satisfying $\gamma^\dual \circ \sigma \circ \gamma^\dual = \gamma^\dual$. 

Such a map $\sigma$ is obtained by setting, for each $\psi\in Z^\dual$ and each $z\otimes b\in Z\otimes^{\h}_A B$, 
\[
\left\langle \left. \sigma(\psi) \, \right|\, z\otimes b \right\rangle_{\C} \coloneqq \left\langle \psi\, \left|\, z\cdot\iota(b)\right. \right\rangle_{\C}
\]
where $z\cdot\iota(b)\in Z^{\dual\dual}$ is the element obtained by embedding $z\in Z$ into $Z^{\dual\dual}$ in the usual way, and then multiplying by the element $\iota(b)\in A^{\dual\dual}$; here $\iota:B\to A^{\dual\dual}$ is a weak expectation, as in Lemma \ref{lem:A-in-B}, and $Z^{\dual\dual}$ is an operator $A^{\dual\dual}$-module as explained in \cite[3.8.9]{BLM}. 

To verify the identity $\gamma^{\dual}\circ \sigma \circ \gamma^{\dual}=\gamma^{\dual}$ one notes that 
\begin{equation}\label{eq:hsh}
(\eta^Z)^\dual \circ \sigma = \id_{Z^\dual}\qquad \text{and}\qquad \sigma\circ \gamma^{\dual}=(\gamma\otimes \id_B)^\dual \circ \tau
\end{equation}
where $\tau:(Z\otimes^{\h}_A B)^{\dual} \to (Z\otimes^{\h}_A B\otimes^{\h}_A B)^{\dual}$ is defined by
\[
\left\langle \left. \tau(\psi) \, \right|\, z\otimes b\otimes b' \right\rangle_{\C} \coloneqq \left\langle \psi\, \left|\, (z\otimes b)\cdot \iota(b')\right.\right\rangle_{\C}.
\]
The identities \eqref{eq:hsh}, along with the fact that $(\gamma\otimes \id_B)\circ \delta=0$ (Lemma \ref{lem:delta-algebra} part (b)), then yield
\[
\gamma^\dual \circ \sigma \circ \gamma^\dual = (\eta^Z)^\dual \circ \sigma \circ \gamma^\dual - \delta^\dual \circ \sigma \circ \gamma^\dual = \gamma^\dual - \delta^\dual \circ (\gamma\otimes \id_B)^\dual \circ \tau=\gamma^\dual
\]
as required.
\end{proof}

\subsection{Proof of Theorem \ref{thm:main}, concluded}\label{subsec:pf-end}
Let $(Z,\zeta)$ be an object of $\Glue(A,\mathfrak F)$, and let $\delta:Z\to Z\otimes^{\h}_A B$ be the map defined in Lemma \ref{lem:delta}. Lemma \ref{lem:delta-algebra} part (b) ensures that the image of $\delta$ is contained in the kernel of $(\eta^Z-\delta)\otimes \id_B$, and Lemma \ref{lem:kernels} identifies the latter kernel with $\functor{G}(Z,\zeta)\otimes^{\h}_A B$. Since $\epsilon^Z\circ \delta$ is the identity on $Z$, by part (a) of Lemma \ref{lem:delta-algebra}, we conclude that the map $\epsilon^Z: \functor{G}(Z,\zeta)\otimes^{\h}_A B\to Z$ is surjective. By Lemma \ref{lem:epsilon} part (c), this implies that there is a unitary natural isomorphism $\functor{P}\functor{G}\cong \id_{\Glue(A,\mathfrak F)}$. We showed in Section \ref{subsec:pf-begin} that there is a unitary natural isomorphism $\functor{G}\functor{P}\cong \id_{\CMod(A)}$, and therefore the functor $\functor{P}:\CMod(A)\to \Glue(A,\mathfrak F)$ is a unitary equivalence.\hfill\qed

\subsection{Gluing data and Hilbert comodules}\label{subsec:comodules}

In this section we briefly explain how Theorem \ref{thm:main} relates to the general machinery of \cite{Crisp-descent}; see that paper for explanations of the undefined terms.

\begin{theorem}\label{thm:comodules}
For each object $(Z,\zeta)$ of $\Glue(A,\mathfrak F)$ let $\widetilde{\delta}:Z\to Z\otimes^{\h}_B (B\otimes^{\h}_A B)$ be the composition
\[
Z \xrightarrow{\delta} Z\otimes^{\h}_A B \xrightarrow[\cong]{z\otimes b\mapsto z\otimes 1\otimes b} Z\otimes^{\h}_B B\otimes^{\h}_A B
\]
where $\delta$ is as in Lemma \ref{lem:delta}. The assignment $(Z,\zeta)\mapsto (Z,\widetilde{\delta})$ extends to a unitary equivalence of $C^*$-categories
\[
\functor{D}:\Glue(A,\mathfrak F) \to \CComod(B\otimes^{\h}_A B)
\]
between $\Glue(A,\mathfrak F)$ and the $C^*$-category of Hilbert comodules over the $C^*$-coalgebra $B\otimes^{\h}_A B$, as defined in \cite[Definition 5.1]{Crisp-descent}. The composite
\[
\functor{D}\circ \functor{P}:\CMod(A) \to \CComod(B\otimes^{\h}_A B)
\]
is unitarily isomorphic to the unitary equivalence $\functor{L}:\CMod(A)\xrightarrow{\cong}\CComod(B\otimes^{\h}_A B)$ of \cite[Theorem 5.6]{Crisp-descent}.
\end{theorem}

\begin{proof}
The fact that $(Z,\widetilde{\delta})$ satisfies the comodule axioms follows from parts (a) and (b) of Lemma \ref{lem:delta-algebra}, plus an additional verification of the Hermitian condition, which is a  consequence of the unitarity of the gluing maps $\zeta_{ij}$. Another straightforward computation shows that every morphism of gluing data is also a morphism of comodules, so we obtain a $*$-functor $\functor{D}$. The canonical inverse $\functor{R}:\CComod(B\otimes^{\h}_A B)\to \CMod(A)$ of the equivalence $\functor{L}$ is given by $(Z,\widetilde{\delta})\mapsto \ker(\eta^Z-\delta)$, and so part (c) of Lemma \ref{lem:delta-algebra} implies that the composite
\[
\Glue(A,\mathfrak F) \xrightarrow{\functor{D}} \CComod(B\otimes^{\h}_A B) \xrightarrow{\functor{R}} \CMod(A)
\]
is equal to the gluing functor $\functor{G}$. Since $\functor{G}$ and $\functor{R}$ are equivalences, $\functor{D}$ is likewise; and we have
\[
\functor{D}\circ \functor{P}\cong \functor{L}\circ \functor{R}\circ \functor{D}\circ \functor{P} \cong \functor{L}\circ \functor{G}\circ \functor{P}\cong \functor{L}.\qedhere
\]
\end{proof}

\begin{remark}
We have proved Theorem \ref{thm:comodules} as a consequence of Theorem \ref{thm:main}. If we had taken the machinery of \cite{Crisp-descent} for granted then we could have proceeded in a slightly more economical way by first proving Theorem \ref{thm:comodules} (still relying on Lemma \ref{lem:delta}), and then deducing Theorem \ref{thm:main} from \cite[Theorem 5.6]{Crisp-descent}. 
\end{remark}

\section{Gluing Morita equivalences and the Picard group}\label{sec:Picard}

In this section we apply Theorem \ref{thm:main} to obtain a generalisation of \cite[Theorem 2.1]{Raeburn-Picard}. We first recall some terminology related to Morita equivalence, following \cite{RW}.

Throughout this section we fix a topological space $T$, and suppose that $A$ and $A'$ are $C^*$-algebras whose primitive ideal spaces have been identified with $T$ via fixed choices of homeomorphisms. 

An \emph{$(A',A)$-equivalence bimodule} is an $A'$-$A$ bimodule $M$ that is simultaneously a right Hilbert $A$-module and a left Hilbert $A'$ module, with the two inner products satisfying 
\[
{}_{A'}\langle m a | n\rangle = {}_{A'}\langle m | na^*\rangle,\quad \langle a'm |n\rangle_A= \langle m |a'^* n\rangle_A,\quad\text{and}\quad {}_{A'}\langle m | n\rangle p = m \langle n| p\rangle_A
\]
for all $m,n,p\in M$, $a\in A$, and $a'\in A'$; and also satisfying
\begin{equation}\label{eq:full}
A'=\overline{\lspan}\left\{ \left. {}_{A'}\langle m|n\rangle\ \right|\ m,n\in M \right\}\  \ \text{and}\ \ A=\overline{\lspan}\left\{\left. \langle m|n\rangle_A\ \right|\ m,n\in M \right\}.
\end{equation}
An isomorphism of $(A',A)$-equivalence bimodules is a bimodule isomorphism that preserves both of the inner products---or, equivalently (thanks to \cite[Theorem 3.5]{Lance} and the fact that the norms induced by the two inner products are equal), an isometric bimodule isomorphism.  

If $M$ is such a  bimodule then for each ideal $J$ of $A$ there is a unique ideal $J'$ of $A'$ with the property that $MJ=J'M$. This correspondence gives rise to a homeomorphism 
\[
h_M:T \xrightarrow{\cong} \Prim A \xrightarrow{J\mapsto J'} \Prim A' \xrightarrow{\cong} T,
\] 
and we call $M$ an  \emph{$(A',A,T)$-equivalence bimodule} if   $h_M=\id_T$.

If $M$ is an $(A',A,T)$-equivalence bimodule, and if $F$ is a closed subset of $T$ corresponding to ideals $J\subset A$ and $J'\subset A'$, then  the restriction $M\restrict_{F}=M/(MJ)=M/(J'M)$ is an $(A'\restrict_F, A\restrict_F,F)$-equivalence bimodule.

We let $\Eq(A',A,T)$ be the category whose objects are $(A',A,T)$-equivalence bimodules and whose morphisms are isomorphisms of equivalence bimodules.

\begin{definition}
Let $\mathfrak F=(F_i)_{i\in I}$ be a locally finite closed cover of $T$. An \emph{$(A',A,\mathfrak F)$-equivalence bimodule} is a pair $(M,\mu)$ consisting of a collection $M=(M_i)_{i\in I}$ of $(A'\restrict_{F_i}, A\restrict_{F_i}, F_i$)-equivalence bimodules, and a collection $\mu=(\mu_{ij})_{(i,j)\in I^2}$ of isomorphisms of $(A'\restrict_{F_ij}, A\restrict_{F_{ij}})$-equivalence bimodules $\mu_{ij}: M_j\restrict_{F_{ij}}\to M_i\restrict_{F_{ij}}$ such that for all $(i,j,k)\in I^3$ we have $\mu_{ij}\restrict_{F_{ijk}} \circ \mu_{jk}\restrict_{F_{ijk}} = \mu_{ik}\restrict_{F_{ijk}}$. An isomorphism $(M,\mu)\to (N,\nu)$ of such bimodules is a collection $(\alpha_i)_{i\in I}$ of isomorphisms of $(A'\restrict_{F_i}, A\restrict_{F_i})$-equivalence bimodules $\alpha_i: M_i\to N_i$ satisfying $\alpha_j\restrict_{F_{ij}}\circ \mu_{ij} = \nu_{ij}\circ \alpha_i\restrict_{F_{ij}}$ for all $(i,j)\in I^2$. We let $\Eq(A',A,\mathfrak F)$ denote the category of $(A',A,\mathfrak F)$-equivalence bimodules and isomorphisms.
\end{definition}

\begin{corollary}\label{cor:Morita}
Let $\mathfrak F=(F_i)_{i\in I}$ be a locally finite closed cover of $T$. The functor 
\[
\functor{P}:\Eq(A',A,T)\to \Eq(A',A,\mathfrak F)
\]
defined on objects by
\[
\functor{P} M \coloneqq \left( \left(M\restrict_{F_i}\right)_{i\in I}, \left(\kappa^M_{ij}\right)_{(i,j)\in I}\right)
\]
and on morphisms by $\functor{P}\alpha \coloneqq \left( \alpha\restrict_{F_i}\right)_{i\in I}$ is an equivalence of categories, with inverse $\functor{G}:\Eq(A',A,\mathfrak F)\to \Eq(A',A,T)$ defined on objects by
\[
\functor{G}(M,\mu)\coloneqq \left\{ (m_i)_{i\in I}\in \bigoplus_{i\in I} M_i\ \left|\ m_i\restrict_{F_{ij}} = \mu_{ij}( m_j\restrict_{F_{ij}})\text{ for all }(i,j)\in I^2 \right. \right\}
\]
and on morphisms by $\functor{G}(\alpha_i)_{i\in I} \coloneqq \bigoplus_{i\in I}\alpha_i$.
\end{corollary}

\begin{proof}
We noted above  that the restriction of an $(A',A,T)$-equivalence bimodule to a closed subset $F\subseteq T$ is an $(A'\restrict_F, A\restrict_F, F)$-equivalence bimodule; this fact ensures that $\functor{P}$ produces $(A',A,\mathfrak F)$-equivalences from $(A',A,T)$-equivalences.

The fact that $\functor{G}$ produces $(A',A,T)$-equivalences from $(A',A,\mathfrak F)$-equivalences is less obvious; let us prove it now. Fix  an $(A',A,\mathfrak F)$-equivalence bimodule $(M,\mu)$, and write $B=\bigoplus_{i\in I} A\restrict_{F_i}$ and $B'= \bigoplus_{i\in I} A'\restrict_{F_i}$.  Since each $M_i$ is an $(A'\restrict_{F_i}, A\restrict_{F_i})$-equivalence bimodule, the direct sum $M=\bigoplus_{i\in I} M_i$ is a $(B', B)$-equivalence bimodule. 

Notice that if we ignore the left $A'\restrict_{F_i}$-module structures then $(M,\mu)$ is an object of $\Glue(A,\mathfrak F)$, and so  Lemma \ref{lem:GZ-over-A} shows that the $B$-valued inner product on $M$, restricted to the closed subspace $\functor{G}(M,\mu)$, makes the latter into a right Hilbert $A$-module. On the other hand, our assumptions on $(M,\mu)$ are entirely symmetrical with respect to the right Hilbert $A\restrict_{F_i}$-module structures vis-\`a-vis the left Hilbert $A'\restrict_{F_i}$-module structures; and so, using the natural left-handed analogues of Definitions \ref{def:Glue} and \ref{def:G},  we may regard $(M,\mu)$ as a gluing datum for left Hilbert $A'$-modules and apply the left-module version of Lemma \ref{lem:GZ-over-A} to conclude that the $B'$-valued inner product on $M$ makes $\functor{G}(M,\mu)$ into a left Hilbert $A'$-module. 
The identities $\langle a'm|n\rangle_A = \langle m|a'^*n\rangle_A$, ${}_{A'}\langle ma |n\rangle = {}_{A'}\langle m|na^*\rangle$, and ${}_{A'}\langle m|n\rangle p = m\langle n|p\rangle_A$ in $\functor{G}(M,\mu)$ all follow immediately from the corresponding identities in the $(B',B)$-equivalence bimodule $M$, so to show that $\functor{G}(M,\mu)$ is an $(A',A)$-equivalence bimodule we just need to verify the fullness condition \eqref{eq:full}. By the $A'$-$A$ symmetry noted above, it will suffice to consider the $A$-valued inner product. 

If
\(
J\coloneqq \overline{\lspan}\left\{ \left. \langle m|n\rangle_A\in A\ \right|\ m,n\in \functor{G}(M,\mu)\right\}
\)
is a proper ideal of $A$ then the ideal $K$ of $B$ generated by $J$ is a proper ideal of $B$: indeed, let $\rho$ be an irreducible representation of $A/J$, let $F_i$ be an element of the covering $\mathfrak F$ having $\ker\rho\in F_i$, and note that  $\rho\circ \pi^B_i$ is a nonzero map on $B$ vanishing on $K$. Now, in the course of proving Theorem \ref{thm:main} we showed that $\functor{G}(M,\mu)\otimes^*_A B \cong M$ as Hilbert $B$-modules, and it is clear from the definition of the $B$-valued inner product on $\functor{G}(M,\mu)\otimes^*_{A} B$ that this inner product takes values in $K$. Since $M$ is a $(B',B)$-equivalence bimodule the range of its $B$-valued inner product spans a dense subspace of $B$, so we must have $K=B$, and therefore $J=A$. 

The last thing to show is that the map $h_{\functor{G}(M,\mu)}: T\to T$ induced by the equivalence bimodule $\functor{G}(M,\mu)$ is the identity. For each $i\in I$ the restriction  of $h_{\functor{G}(M,\mu)}$ to the subset $F_i$ is equal to the map $h_{\functor{G}(M,\mu)\restrict_{F_i}}$ induced by the $(A'\restrict_{F_i}, A\restrict_{F_i})$-equivalence bimodule $\functor{G}(M,\mu)$ (cf.~\cite[Corollary 3.33(b)]{RW}). Our proof of Theorem \ref{thm:main} showed that the map
\[
\epsilon^M: \functor{G}(M,\mu)\otimes^{\h}_A B \xrightarrow{m\otimes b\mapsto mb} M
\]
is a unitary isomorphism of Hilbert $B$-modules, and so it restricts for each $i\in I$ to a unitary isomorphism of Hilbert $A\restrict_{F_i}$-modules $\functor{G}(M,\mu)\restrict_{F_i} \cong M_i$. Since $\epsilon^M$ is obviously $A'$-linear, the latter isomorphism is an isomorphism of $(A'\restrict_{F_i},A\restrict_{F_i})$-equivalence bimodules, and so $h_{\functor{G}(M,\mu)\restrict_{F_i}} = h_{M_i}$ as maps $F_i\to F_i$. We assumed that $M_i$ is an $(A'\restrict_{F_i}, A\restrict_{F_i}, F_i)$-equivalence bimodule, so $h_{M_i}=\id_{F_i}$. Thus the map $h_{\functor{G}(M,\mu)}$ restricts to the identity map on each $F_i$, and since these sets cover $T$ we conclude that $h_{\functor{G}(M,\mu)}=\id_T$.

We have now shown that  $\functor{G}$ does indeed send $\Eq(A',A,\mathfrak F)$ into $\Eq(A',A,T)$. It is easy to see that $\functor{P}$ and $\functor{G}$ are functors, and so we are left to show that they are inverses. 

For each object $M$ of $\Eq(A',A,T)$, disregarding the left $A'$-module structure, Theorem \ref{thm:main} (and its proof) yields a natural unitary isomorphism of right Hilbert $A$-modules $\Phi^M: M\xrightarrow{\cong} \functor{G}\functor{P}M$. Since the left $A'$-module structure comes from a $*$-homomorphism $A'\to \Hom_{\CMod(A)}(M,M)$, the naturality of $\Phi^M$ ensures that this isomorphism is also an isomorphism of left $A'$-modules, and so $\functor{G}\functor{P}\cong \id$ on the category $\Eq(A',A,T)$. An analogous argument shows that for each object $(M,\mu)$ of $\Eq(A',A,\mathfrak F)$ the natural unitary isomorphism 
\[
\epsilon^M:\functor{P}\functor{G}(M,\mu) \xrightarrow{\cong_{\Glue(A,\mathfrak F)} } (M,\mu)
\]
 of Theorem \ref{thm:main} is in fact an isomorphism in $\Eq(A',A,\mathfrak F)$, and we conclude that $\functor{P}$ and $\functor{G}$ are mutually inverse equivalences.
\end{proof}

\begin{definition}
Let $A$ be a $C^*$-algebra with $\Prim A\cong T$. The \emph{Picard group} $\Pic_T A$ is the group of isomorphism classes of $(A,A,T)$-equivalence bimodules, with the group operation being the tensor product $[M]\cdot [N] \coloneqq \left[ M\otimes^*_A N\right]$.
\end{definition}

\begin{corollary}\label{cor:Picard}
Let $A$ and $A'$ be $C^*$-algebras with $\Prim A\cong \Prim A'\cong T$. Suppose that there exist a locally finite closed covering $\mathfrak F=(F_i)_{i\in I}$ of $T$, a collection $(N_i)_{i\in I}$ of $(A'\restrict_{F_i}, A\restrict_{F_i}, F_i)$-equivalence bimodules, and a collection $(\nu_{ij})_{(i,j)\in I^2}$ of unitary isomorphisms of $(A'\restrict_{F_{ij}}, A\restrict_{F_{ij}})$-bimodules $\nu_{ij}: N_j\restrict_{F_{ij}}\to N_i\restrict_{F_{ij}}$. Then $\Pic_T A\cong \Pic_T A'$.
\end{corollary}

\begin{remark}
If we added the requirement that $\nu_{ik}\restrict_{F_{ijk}}= \nu_{ij}\restrict_{F_{ijk}}\circ \nu_{jk}\restrict_{F_{ijk}}$ then Corollary \ref{cor:Morita} would imply that $A$ and $A'$ are Morita equivalent relative to $T$, and so of course $\Pic_T A \cong \Pic_T A' $. In general the $2$-cocycle $\nu_{ij}\restrict_{F_{ijk}}\circ \nu_{jk}\restrict_{F_{ijk}}\circ \nu_{ik}\restrict_{F_{ijk}}^*$ is an obstruction to gluing the local Morita equivalences $N_i$ into a global one, but we still obtain an isomorphism of Picard groups. 
\end{remark}

\begin{proof}[Proof of Corollary \ref{cor:Picard}]
For each $i\in I$ let $\widetilde{N}_i$ be the $(A\restrict_{F_i},A'\restrict_{F_i},F_i)$-equivalence bimodule that is dual to $N_i$, and for each $(i,j)$ let $\widetilde{\nu}_{ij}: \widetilde{N}_j\restrict_{F_{ij}} \xrightarrow{\cong} \widetilde{N}_i\restrict_{F_{ij}}$ be the unitary isomorphism that is dual to $\nu_{ij}$, in the sense described, e.g., in \cite[p. 49]{RW}. We claim that the assignment
\[
\left( \left( M_i \right)_{i\in I}, \left(\mu_{ij}\right)_{(i,j)\in I^2}\right) \mapsto \left( \left( \widetilde{N}_i\otimes^*_{A'\restrict_{F_i}} M_i\otimes^*_{A'\restrict_{F_i}} N_i\right)_{i\in I}, \left( \widetilde{\nu}_{ij}\otimes \mu_{ij}\otimes \nu_{ij}\right)_{(i,j)\in I^2}\right)
\]
extends to an equivalence of categories $\functor{N}: \Eq(A',A',\mathfrak F) \xrightarrow{\cong} \Eq(A,A,\mathfrak F)$, given on morphisms by $\functor{N} (\alpha_i)_{i\in I}\coloneqq (\id_{\widetilde{N}_i}\otimes \alpha_i\otimes \id_{N_i})_{i\in I}$. 

To prove this we first recall that the restriction functors $\restrict_F$ are compatible with tensor products and with taking duals, in the sense that, for instance, $\widetilde{N}_i\restrict_{F_{ij}}$ is the dual of $N_i\restrict_{F_{ij}}$, and
\[
\left( \widetilde{N}_i \otimes^*_{A'\restrict_{F_i}} M_i \otimes^*_{A'\restrict_{F_i}} N_i\right) \restrict_{F_{ij}} \cong \widetilde{N}_i\restrict_{F_{ij}}\otimes^*_{A'\restrict_{F_{ij}}} M_i\restrict_{F_{ij}}\otimes^*_{A'\restrict_{F_{ij}}} N_i\restrict_{F_{ij}}
\] via the factor-wise restriction maps. 

To see that the proposed formula for $\functor{N}(M,\mu)$ actually gives an object of $\Eq(A,A,F)$ we must verify that
\begin{multline}\label{eq:Picard-proof}
\left( \widetilde{\nu}_{ij}\restrict_{F_{ijk}} \circ \widetilde{\nu}_{jk}\restrict_{F_{ijk}}\circ \widetilde{\nu}_{ik}\restrict_{F_{ijk}}^*\right) \otimes \left(\mu_{ij}\restrict_{F_{ijk}}\circ \mu_{jk}\restrict_{F_{ijk}}\circ \mu_{ik}\restrict_{F_{ijk}}^*\right) \\ \otimes \left( \nu_{ij}\restrict_{F_{ijk}}\circ \nu_{jk}\restrict_{F_{ijk}} \circ \nu_{ik}\restrict_{F_{ijk}}^*\right)
\end{multline}
is the identity on $\widetilde{N}_i\restrict_{F_{ijk}}\otimes^*_{A'\restrict_{F_{ijk}}} M_i\restrict_{F_{ijk}} \otimes^*_{A\restrict_{F_{ijk}}} N_i\restrict_{F_{ijk}}$. The middle tensor-factor in \eqref{eq:Picard-proof} is the identity on $M_i\restrict_{F_{ijk}}$ because $(M,\mu)$ is an object of $\Eq(A',A',\mathfrak F)$. The right-most tensor-factor in \eqref{eq:Picard-proof} is a unitary bimodule automorphism of the $(A'\restrict_{F_{ijk}}, A\restrict_{F_{ijk}}, F_{ijk})$-equivalence bimodule $N_i\restrict_{F_{ijk}}$, and therefore it is given by multiplication---either on the left or on the right (cf.~\cite[Proposition 5.7(a)]{RW})---by some unitary $f\in C_b(F_{ijk})\cong \Centre\Multiplier(A'\restrict_{F_{ijk}}) \cong \Centre\Multiplier(A\restrict_{F_{ijk}})$. The left-hand tensor-factor in \eqref{eq:Picard-proof}, being the dual of the right-hand factor, is multiplication by $f^*$, and since the tensors are balanced over $A'\restrict_{F_{ijk}}$---and hence (thanks to nondegeneracy) over the multiplier algebra of $A'\restrict_{F_{ijk}}$---we find that \eqref{eq:Picard-proof} is equal to $f^*\otimes \id \otimes f = \id\otimes f^*f\otimes \id = \id\otimes \id\otimes \id$ as required.

It is easy to check that $\functor{N}$ is a functor. Applying the same construction with $N_i$ replaced by $\widetilde{N}_i$, and $\nu_{ij}$ replaced by $\widetilde{\nu}_{ij}$, we obtain another functor $\widetilde{\functor{N}}:\Eq(A,A,\mathfrak F)\to \Eq(A',A',\mathfrak F)$, and the argument of \cite[Lemma 2.10]{Raeburn-Picard} shows that $\functor{N}$ and $\widetilde{\functor{N}}$ are mutually inverse equivalences.

Now $\Pic_T A$ is the set of isomorphism classes in the category $\Eq(A,A,T)$, and Corollary \ref{cor:Morita} identifies this category with $\Eq(A,A,\mathfrak F)$; likewise for $\Pic_T A'$ and $\Eq(A',A',\mathfrak F)$. Thus the equivalence $\functor{N}$ induces a bijection of sets $\Pic_T A' \xrightarrow{\cong}\Pic_T A$, and a computation as in \cite[Proposition 2.12]{Raeburn-Picard} shows that this bijection is a group isomorphism.
\end{proof}

\bibliographystyle{alpha}
\bibliography{descent}

\end{document}